\documentclass[12pt]{article}

\usepackage{latexsym}
\usepackage{amssymb}
\usepackage{amsmath}
\usepackage{amsthm}

\newtheorem{theorem}{Theorem}[section]
\newtheorem{lemma}[theorem]{Lemma}
\newtheorem{remark}[theorem]{Remark}
\newtheorem{corollary}[theorem]{Corollary}
\newtheorem{proposition}[theorem]{Proposition}

\newtheorem{definition}[theorem]{Definition}
\newcommand{\bd}[1]{\begin{definition}\label{#1}\rm}
\newcommand{\ed}{\end{definition}}
\newcommand{\bt}[1]{\begin{theorem}\label{#1}}
\newcommand{\et}{\end{theorem}}
\newcommand{\bprop}[1]{\begin{proposition}\label{#1}}
\newcommand{\eprop}{\end{proposition}}
\newcommand{\bcor}[1]{\begin{corollary}\label{#1}}
\newcommand{\ecor}{\end{corollary}}

\newcommand{\D}{\displaystyle}
\newcommand{\T}{\textstyle}
\newcommand{\lra}{\longrightarrow}

\newcommand{\stack}[2]{\raisebox{-2pt} 
{\renewcommand{\arraystretch}{.01} 
\begin{tabular}{c} 
$#2$\\$\scriptscriptstyle #1$ 
\end{tabular} 
}}

\newcommand{\vp}{\varphi}

\def\1{\, {\rm I}\mskip-10mu 1} 
\,

\renewcommand{\t}[1]{\tilde{#1}} 
\DeclareMathOperator*{\esssup}{ess\,sup}
\DeclareMathOperator*{\essinf}{ess\,inf}

\begin{document}
\title{Boundedness of the Stationary Solution to the Boltzmann Equation with 
Spatial Smearing, Diffusive Boundary Conditions, and Lions' Collision Kernel} 
\par
\author{J\"org-Uwe L\"obus
\\ Matematiska institutionen \\ 
Link\"opings universitet \\ 
SE-581 83 Link\"oping \\ 
Sverige 
}
\date{}
\maketitle
{\footnotesize
\noindent
\begin{quote}
{\bf Abstract} We investigate the Boltzmann equation with spatial smearing, 
diffusive boundary conditions, and Lions' collision kernel. Both, the 
physical as well as the velocity space, are assumed to be bounded. Existence 
and uniqueness of a stationary solution, which is a probability density, has 
been demonstrated in \cite{CPW98} under a certain smallness assumption on 
the collision term. We prove that whenever there is a stationary solution then 
it is a.e. positively bounded from below and above. 
\noindent 



\end{quote}
}

\section{Introduction}\label{sec:1}
\setcounter{equation}{0} 

Over the past two decades, there has been a recurring interest in bounds on 
solutions to various forms of the Boltzmann equation. The different assumptions 
on the Boltzmann equation have led to different approaches. 

For example, in \cite{Fo01} stochastic calculus (Malliavin calculus) is used to 
establish boundedness from below of solutions to a certain form of the Boltzmann 
equation. Furthermore, in \cite{MV04} and \cite{Mo05} bounds in several function 
spaces and bounds by particular functions are derived by means of a detailed 
analysis of the collision kernel. In particular, regularity properties of the 
gain term are investigated. In \cite{BG17} and \cite{GBV09} sophisticated 
comparison principles are provided in order to establish Maxwellian bounds. 
\medskip 

The present paper uses a mathematical description of a rarefied gas in a vessel 
with diffusive boundary conditions introduced in \cite{CPW98}. The conditions on 
this form of the Boltzmann equation are physically motivated and allow to 
demonstrate the existence of a unique probability density which is a stationary 
solution to this equation. In particular, boundedness of the physical space 
$\Omega$, i. e. diam$(\Omega):=\sup\{|r_1-r_2|:r_1,r_2\in\Omega\}<\infty$, and 
strict positivity of the modulus of the velocity $v$ of a particle, i. e. $0< 
v_{min}<|v|$ imply that for all particles the free crossing time through the 
vessel is bounded by diam$(\Omega)/v_{min}$. Our analysis uses this assumption in 
(\ref{3.9}), (\ref{3.19}), and (\ref{4.28}). In addition our analysis relies 
on the physically relevant hypothesis $|v|<v_{max}<\infty$. This assumption is 
crucial for the proof of Lemma \ref{Lemma4.1} below. 

From \cite{CPW98} and other references cited and discussed in \cite{CPW98} we 
take over the presence of spatial smearing in the collision operator and the 
boundedness of the collision kernel. These features entail the existence of a 
unique stationary solution which is a probability density, as demonstrated in 
the proof of Theorem 2.2 of \cite{CPW98}. 

\medskip 
The objective of the present paper is to show that whenever a Boltzmann equation 
in a form adapted from \cite{CPW98} and \cite{Li94} has a stationary solution then 
it is a.e. positively bounded from below and above. The proof consists of two basic 
technical steps. 

The first one is to provide suitable bounds on the gain term. This is carried out 
in Section 3 by an iteration based on J.-L. Lion's regularity result in \cite{Li94}, 
Theorem IV.1 and the subsequent Remark ii). For this reason we use the particular 
form of the collision kernel in \cite{Li94}. 

The second basic technical step in order to establish bounds on the stationary 
solution $g\equiv g(y,v)$, $y\in\overline{\Omega}$, $v\in V$, refers to the flux 
$J(g)$ in physical boundary points $r\in\partial\Omega$ and the total mass $\|g 
(y,\cdot)\|_{L^1(V)}$ in physical inner points $y\in\Omega$, the dot refers to 
integration over all possible velocities. In the proof of Lemma \ref{Lemma4.1} 
below, we introduce an iteration to demonstrate that the function $\Omega\ni y 
\mapsto\|g(y,\cdot)\|_{L^1(V)}$ belongs to $L^q(\Omega)$ for all $1\le q<\infty$ 
and that $J(g)\in L^q(\partial\Omega)$ for all $1\le q<\infty$.

The results of the two basic steps then yield the a.e. positive lower and upper 
bounds on the stationary solution $g$, see Lemma \ref{Lemma4.3} and Theorem 
\ref{Theorem4.4}. 

\section{Boltzmann Equation with Spatial Smearing and Diffusive Boundary 
Conditions}\label{sec:2}
\setcounter{equation}{0}

Let $\Omega\subset{\mathbb R}^d$, $d=2,3$, be a bounded open set with smooth 
boundary, called the physical space. Furthermore, let $V:=\{v\in {\mathbb R}^d: 
0<v_{min}<|v|<v_{max}<\infty\}$ be the velocity space and let $\lambda>0$. Denote 
by $n(r)$ the outer normal at $r\in\partial\Omega$ and indicate the inner product 
in ${\mathbb R}^d$ by ``$\circ$". Following \cite{CPW98}, for $(r,v,t)\in\Omega 
\times V\times [0,\infty)$, consider the {\it Boltzmann equation} 
\begin{eqnarray*}
\frac{d}{dt}\, p(r,v,t)=-v\circ\nabla_rp(r,v,t)+\lambda Q(p,p)\, (r,v,t)
\end{eqnarray*} 
with {\it diffusive boundary conditions} 
\begin{eqnarray}\label{2.1} 
p(r,v,t)=J(r,t)(p)M(r,v)\, ,\quad r\in\partial\Omega,\ v\circ n(r)\le 0, 
\end{eqnarray} 
and initial condition $p(0,r,v):=p_0(r,v)$ or its {\it integrated (mild) version} 
\begin{eqnarray}\label{2.2} 
p(r,v,t)=S(t)\, p_0(r,v)+\lambda\int_0^t S(t-s)\, Q(p,p)\, (r,v,s)\, ds\, , 
\end{eqnarray} 
where we specify the following. 
\begin{itemize} 
\item[(i)] The {\it flux} $J$ is given by 
\begin{eqnarray*} 
J(r,t)(p)=\int_{\{v\in V:v\circ n(r)\ge 0\}}v\circ n(r)\, p(r,v,t)\, dv\, , 
\quad r\in\partial\Omega,\ t\ge 0. 
\end{eqnarray*} 
\item[(ii)] 
The function $M$ defined on the set $\{(r,v):r\in\partial\Omega,\ v\in V,\ 
v\circ n(r)\le 0\}$ is continuous, has positive lower and upper bounds, and 
satisfies 
\begin{eqnarray*} 
\int_{\{v\in V:v\circ n(r)\le 0\}}|v\circ n(r)|\, M(r,v)\, dv=1\, . 
\end{eqnarray*} 
\item[(iii)] The semigroup $S(t)$, $t\ge 0$, in $L^1(\Omega\times V)$, called 
the {\it Knudsen semigroup}, is formally the solution to the initial boundary 
value problem 
\begin{eqnarray*} 
\left(\frac{d}{dt} + v\circ\nabla_r\right)(S(t)p_0)(r,v)=0\, , 
\end{eqnarray*} 
\begin{eqnarray*} 
(S(t)p_0)(r,v)=J(r,t)(S(\cdot)p_0)M(r,v)\, ,\quad r\in \partial\Omega,\ v\circ 
n(r)\le 0. 
\end{eqnarray*} 
\item[(iv)] Denoting by $\chi$ the indicator function and setting $p:=0$ on 
$\Omega\times ({\mathbb R}^d\setminus V)\times [0,\infty)$, the {\it collision 
operator} $Q$ is given by 
\begin{eqnarray*}
&&\hspace{-.5cm}Q(p,p)(r,v,t)=\int_{\Omega}\int_V\int_{S_+^{d-1}}B(v,v_1,e) 
h_\gamma(r,y)\chi_{\{(v^\ast ,v_1^\ast)\in V\times V\}}(v,v_1,e)\times \\ 
&&\hspace{.5cm}\times\left(p(r,v^\ast ,t)p(y,v_1^\ast ,t)-p(r,v,t)p(y,v_1,t) 
\vphantom{l^1}\right)\, de\, dv_1\, dy\, .\vphantom{\int} 
\end{eqnarray*} 
Here $S^{d-1}$ is the unit sphere. In addition, $S^{d-1}_+\equiv S^{d-1}_+(v-v_1) 
:=\{e\in S^{d-1}:e\circ(v-v_1)>0\}$, $v^\ast :=v-e\circ (v-v_1)\, e$, $v_1^\ast:= 
v_1+e\circ (v-v_1)\, e$ for $e\in S^{d-1}_+$ as well as $v,v_1\in V$, and $de$ 
refers to the normalized Riemann-Lebesgue measure on $S^{d-1}_+$.  
\item[(v)] The {\it collision kernel} $B\equiv B(v,v_1,e)$ is non-negative, 
bounded, continuous on $V\times V\times S^{d-1}$, and symmetric in $v$ and 
$v_1$. It satisfies $B(v^\ast,v_1^\ast,e)=B(v,v_1,e)$ for all $v,v_1\in V$ and 
$e\in S^{d-1}_+$ for which $(v^\ast,v_1^\ast)\in V\times V$. 
\item[(vi)] The {\it smearing} function $h_\gamma$ is continuous on $\overline{ 
\Omega\times\Omega}$, is non-negative and symmetric, and vanishes for $|r-y|\ge 
\gamma>0$.  
\end{itemize} 
\begin{remark}\label{Remark2.1} 
{\rm For $t\ge 0$, let us regard 
\begin{eqnarray*}
\Omega\times V\times [0,t]\ni (r,v,s)\mapsto S(t-s)\, Q(p,p)(r,v,s) 
\end{eqnarray*} 
as a measurable function. Recall also that for every $t\ge 0$, $S(t):L^1(\Omega 
\times V)\mapsto L^1(\Omega\times V)$ has operator norm 1. Noting that $L^1( 
\Omega\times V)$ is separable, the integral in (\ref{2.2}) is a well-defined 
{\it Bochner integral} whenever } 
\begin{eqnarray*} 
\int_0^t \|Q(p,p)(\cdot,\cdot,s)\|_{L^1(\Omega\times V)}\, ds<\infty\, . 
\end{eqnarray*} 
\end{remark}
\begin{remark}\label{Remark2.2} 
{\rm We mention that the map 
\begin{eqnarray*}
{\mathbb R}^{2d}\ni (v,v_1)\mapsto (v^\ast,v_1^\ast):=\left(v-e\circ (v-v_1)\, e 
\, ,\, v_1+e\circ (v-v_1)\, e\vphantom{l^1}\right) 
\end{eqnarray*} 
has for fixed $e\in S^{d-1}_+$ an inverse which we denote by $(v^{-\ast},v_1^{- 
\ast})$. It is given by the relation 
\begin{eqnarray*}
{\mathbb R}^{2d}\ni (v,v_1)=\T\left(v^\ast-e\circ (v^\ast-v_1^\ast)\, e\, ,\, 
v_1^\ast+e\circ (v^\ast-v_1^\ast)\, e\right)\, .  
\end{eqnarray*} 
In addition we mention that the absolute value of the Jacobian determinant of the 
map $(v,v_1)\mapsto (v^\ast,v_1^\ast)$ is one. }
\end{remark} 

\section{Analysis under Lions' Assumptions on the Collision Kernel} 
\label{sec:3}
\setcounter{equation}{0}

In this section we examine the stationary solution to the equation (\ref{2.2}) 
under conditions that allow us to use results of \cite{Li94}. In particular, 
we are interested in a certain upper bound on the gain term, see Lemma 
\ref{Lemma3.4} below. Introduce $\theta:=\arccos\left(e\circ (v-v_1)/|v-v_1| 
\right)$. Since we always suppose $e\in S^{d-1}_+(v-v_1)$ we have $\theta\in [0, 
\pi/2)$. In order to be compatible with \cite{Li94}, throughout this section we 
shall suppose that for all $(v,v_1)\in V\times V$ and $e\in S^{d-1}_+(v-v_1)$ 
\begin{eqnarray}\label{3.1}
B(v,v_1,e)={\bf B}(|v-v_1|,\theta)\quad\mbox{\rm for some}\quad{\bf B}\in 
C_c^\infty((0,\infty)\times (0,\pi/2))\, , 
\end{eqnarray} 
the space of all infinitely differentiable real functions with compact 
support contained in $(0,\infty)\times (0,\pi/2)$. We mention that $B$ 
defined in this way satisfies (v) of Section \ref{sec:2}. 

According to \cite{CPW98}, Theorem 2.2, there is a $\lambda_0>0$ such 
that for $\lambda\le\lambda_0$ we have the following. There exists a 
non-negative almost everywhere on $\overline{\Omega}\times V$ defined real 
function $g\equiv g(\lambda)$ with $\|g\|_{L^1(\Omega\times V)}=1$ such 
that $g_s(\cdot,\cdot,t):=g$ is the unique non-negative stationary solution 
to (\ref{2.2}) with $\|g_s(\cdot,\cdot,t)\|_{L^1(\Omega\times V)}=1$, 
$t\ge 0$. The bound $\lambda_0$ is determined by (2.38), (2.17), and (3.9) 
of \cite{CPW98}. See also the remark after the proof of Theorem 2.2 in 
\cite{CPW98}. 

By the stationarity of $g_s(\cdot,\cdot,t)$ it is customary to write $Q(g,g) 
(r,v)$ instead of $Q(g_s,g_s)$ $(r,v,t)$, $t\ge 0$. Because of 
\begin{eqnarray*}
\int_0^t S(s)\, Q(g,g)(r,v)\, ds=\int_0^t S(t-s)\, Q(g,g)(r,v)\, ds 
\end{eqnarray*} 
we have 
\begin{eqnarray}\label{3.2}
g(r,v)=S(t)\, g(r,v)+\lambda\int_0^t S(s)\, Q(g,g)(r,v)\, ds\, ,\quad t\ge 0, 
\end{eqnarray} 
for the precise meaning see Remark \ref{Remark3.2} below. 
\medskip

Let us introduce  
\begin{eqnarray*}
\hat{B}(v,v_1):=\int_{S_+^{d-1}}B(v,v_1,e)\cdot\chi_{\{(v^\ast,v_1^\ast)\in V 
\times V\}}(v,v_1,e)\, de\, ,\quad (v,v_1)\in V\times V,   
\end{eqnarray*} 
and 
\begin{eqnarray*}
g_\gamma(r,v_1):=\int_{y\in\Omega}g(y,v_1)\, h_\gamma(r,y)\, dy\, ,\quad (r,v_1) 
\in\overline{\Omega}\times V. 
\end{eqnarray*} 
We define now 
\begin{eqnarray*}
\hat{B}_g(r,v):=\lambda\int_V\hat{B}(v,v_1)\, g_\gamma(r,v_1)\, dv_1\, ,\quad 
(r,v)\in\overline{\Omega}\times V. 
\end{eqnarray*} 
Furthermore, for $r\in\overline{\Omega}$ let us consider $g(r,\cdot)$ and 
$g_\gamma(r,\cdot)$ as functions defined on ${\mathbb R}^d$ by extending 
them by zero outside of $V$. Recalling the notation of Section \ref{sec:2}, 
in this section we shall use the decomposition $Q(g,g)=Q^+(g,g)-Q^-(g,g)$ 
of the collision operator specified by 
\begin{align}\label{3.3}
\begin{split}
&Q^+(g,g)(r,v)\vphantom{\dot{f}} \\ 
&\hspace{1.0cm}=\int_V\int_{S_+^{d-1}}B(v,v_1,e)g(r,v^\ast)g_\gamma(r,v_1^\ast) 
\chi_{\{(v^\ast,v_1^\ast)\in V\times V\}}(v,v_1,e)\, de\, dv_1 \\ 
&\hspace{1.0cm}=\int_{{\mathbb R}^d}\int_{S_+^{d-1}}B(v,v_1,e)g(r,v^\ast)g_\gamma 
(r,v_1^\ast)\, de\, dv_1\, . 
\end{split} 
\end{align}
In fact, we have $\lambda Q^-(g,g)(r,v)=g(r,v)\hat{B}_g(r,v)$ where 
\begin{eqnarray}\label{3.4}
\hat{B}_g(r,v)\le\lambda\|h_\gamma\|\|B\|\|g(\cdot,\cdot,t)\|_{L^1 
(\Omega\times V)}=\lambda\|h_\gamma\|\|B\|\, . 
\end{eqnarray} 
\begin{remark}\label{Remark3.1}
{\rm In other words, $\hat{B}_g$ is bounded on $\overline{\Omega}\times V$. 
Moreover, by (vi), the map $\overline{\Omega}\ni r\mapsto g_\gamma(r,\cdot )$ 
is bounded and uniformly continuous in $L^1(V)$. Thus by (v), $\hat{B}_g$ is 
bounded and continuous on $\overline{\Omega}\times V$. }
\end{remark}
\begin{remark}\label{Remark3.2} 
{\rm It follows from Remark \ref{Remark2.2} that $\int_\Omega\int_V Q(g,g) 
(r,v)\, dv\, dr=0$. Thus we have 
\begin{align}\label{3.5}
\begin{split}
&\hspace{-.5cm}\|S(u)Q(g,g)\|_{L^1(\Omega\times V)}\le\|Q(g,g)\|_{L^1(\Omega 
\times V)}\le 2\|Q^-(g,g)\|_{L^1(\Omega\times V)} \\ 
&\hspace{.5cm}=\frac{2}{\lambda}\|g\hat{B}_g\|_{L^1(\Omega\times V)}\le 2\| 
h_\gamma\|\|B\|<\infty\, ,\quad u\ge 0,
\end{split} 
\end{align} 
where, for the last estimate, we have taken into consideration $\|g\|_{L^1( 
\Omega\times V)}=1$ and we have applied (\ref{3.4}). Recalling Remark 
\ref{Remark2.1}, for $t\ge 0$ we may regard $\Omega\times V\times [0,t]\ni 
(r,v,u)\mapsto S(u)\, Q(g,g)(r,v,u)$ as a measurable function. By (\ref{3.5}) 
and the separability of $L^1(\Omega\times V)$ the integral in (\ref{3.2}) is 
a Bochner integral. Furthermore, according to \cite{Ja97}, Appendix C, we may 
evaluate the integral a.e. on $\Omega\times V$. }
\end{remark}

For $(r,v)\in\overline{\Omega}\times V$ we will use the notation $T_\Omega 
\equiv T_\Omega(r,v):=\inf\{s>0:r-sv\not\in\Omega\}$, the first exit time from 
$\Omega$ of $[0,\infty)\ni t\mapsto r-tv$. Observe that for $(r,v)\in\partial 
\Omega\times V$ with $v\circ n(r)>0$ we have $T_\Omega(r,v)>0$ and that for 
$(r,v)\in\partial\Omega\times V$ with $v\circ n(r)<0$ it holds that $T_\Omega 
(r,v)=0$. 
\medskip

For $(y,v)\in\overline{\Omega}\times V$ let $y^-\equiv y^-(y,v):=y-T_{\Omega} 
(y,v)\, v$. Note that $y^-\in\partial\Omega$. Likewise, for $(r,v)\in\overline 
{\Omega}\times V$ define $r^-$. Furthermore, introduce $r^+\equiv r^+(r,v):= 
r^-(r,-v)$ and observe that $r^-,r^+\in\partial\Omega$ Let us also recall the 
definition of $J$ in (i). Because of the stationarity of $g$ in the sense of 
(\ref{3.2}), $J(y,t)(g)$ is constant in the second argument $t$. We shall 
therefore write $J(y,\cdot)(g)$. 
\begin{lemma}\label{Lemma3.3} 
Let $g$ satisfy (\ref{3.2}) in the sense of Remark \ref{Remark3.2} and 
let $(r,v)\in \Omega\times V$ or $(r,v)\in\partial\Omega\times V$ such that 
$v\circ n(r)\ge 0$. Then 
\begin{eqnarray*}
1\le\psi_g(r,v,t):=\exp\left\{\int_0^t\hat{B}_g(r-sv,v)\, ds\right\}\le\sup 
\psi_g<\infty\, ,\quad t\in [0,T_{\Omega}(r,v)], 
\end{eqnarray*} 
where the supremum is taken over $\{(r,v,t):(r,v)\in\Omega\times V,\ t\in 
[0,T_{\Omega}(r,v)]\}$. Suppose $g(r,v)<\infty$. Then 
\begin{align}\label{3.6}
\begin{split}
&g(r-tv,v) \\ 
&\hspace{.5cm}=\psi_g(r,v,t)\left(-\int_0^t\frac{\lambda Q^+(g,g)(r-sv,v)} 
{\psi_g(r,v,s)}\, ds+g(r,v)\right)\, ,\quad t\in [0,T_{\Omega}(r,v)]. 
\end{split} 
\end{align} 
\end{lemma} 
\begin{proof} As already mentioned in Remark \ref{Remark3.1}, $\hat{B}_g$ is 
bounded and continuous on $\overline {\Omega}\times V$. In particular, 
$[0,T_\Omega]\ni t\mapsto\hat{B}_g(r-tv,v)$ is continuous for all $(r,v) 
\in\Omega\times V$ or $(r,v)\in\partial\Omega\times V$ with $v\circ n(r) 
\ge 0$. According to (\ref{3.2}) and Remark \ref{Remark3.2} we have 
for a.e. $(r,v)\in\Omega\times V$ and $t\in [0,T_{\Omega}]$, and hence 
also for a.e. $(r,v)\in\partial\Omega\times V$ with $v\circ n(r)\ge 0$ 
and $t\in [0,T_{\Omega}]$
\begin{align}\label{3.7}
\begin{split}
&\hspace{-.5cm}g(r-tv,v)-g(r,v)=-\int_0^t\lambda Q(g,g)(r-sv,v)\, ds \\ 
&\hspace{.5cm}=-\int_0^t\lambda Q^+(g,g)(r-sv,v)\, ds+\int_0^t\hat{B}_g 
(r-sv,v)\, g(r-sv,v)\, ds  
\end{split} 
\end{align} 
whenever $g(r,v)<\infty$. Again by the properties of $\hat{B}_g$ collected in 
Remark \ref{Remark3.1}, the related homogeneous equation $\frac{d}{dt}\varphi 
(t)=\hat{B}_g(r-tv,v)\varphi(t)$, $t\in [0,T_{\Omega}]$, with initial value 
$\varphi(0)=g(r,v)$ has the unique solution 
\begin{eqnarray}\label{3.8}
\quad\varphi(t)=g(r,v)\, \psi_g(r,v,t)\equiv g(r,v)\exp\left\{\int_0^t\hat{B}_g 
(r-sv,v)\, ds\right\}\, ,\quad  t\in [0,T_{\Omega}],  
\end{eqnarray} 
whenever $g(r,v)<\infty$. En passant we note that, by (\ref{3.4}), we have 
$1\le\psi_g$ and 
\begin{align}\label{3.9}
\begin{split}
&\hspace{-.5cm}\sup\psi_g\le\exp\left\{\sup_{(y,v)\in\overline{\Omega}\times 
V}T_\Omega(y,v)\cdot\lambda\|h_\gamma\|\|B\|\right\} \\ 
&\hspace{.5cm}\le \exp\left\{\frac{{\rm diam}(\Omega)}{v_{min}}\cdot\lambda\| 
h_\gamma\|\|B\|\right\}<\infty 
\end{split} 
\end{align} 
where the supremum on the left-hand side is taken over $\{(r,v,t):(r,v)\in\Omega 
\times V,\ t\in [0,T_{\Omega}(r,v)]\}$. 
\medskip

Now recall (\ref{3.7}) and keep in mind  uniqueness of the related homogeneous 
equation. For $(r,v)\in \Omega\times V$ or $(r,v)\in\partial\Omega\times V$ with 
$v\circ n(r)\ge 0$ there is a unique solution to 
\begin{eqnarray*}
&&f(r-tv,v)-f(r,v) \\ 
&&\hspace{.5cm}=-\int_0^t\lambda Q^+(g,g)(r-sv,v)\, ds+\int_0^t\hat{B}_g(r-sv,v) 
\, f(r-sv,v)\, ds\, , 
\end{eqnarray*} 
$t\in [0,T_{\Omega}]$, under the initial condition $f(0)=g(r,v)$ whenever $g(r,v) 
<\infty$. This solution is representable as the left-hand side as well as the 
right-hand side of (\ref{3.6}). 
\end{proof}
\medskip

\begin{lemma}\label{Lemma3.4} 
Let $g$ satisfy (\ref{3.2}) in the sense of Remark \ref{Remark3.2} and 
suppose (\ref{3.1}). We have $g_\gamma\in L^\infty(\Omega\times V)$, note also 
Remark \ref{Remark3.1}. Furthermore, there is a constant $0<c_Q<\infty$ 
independent of $(y,v)\in\Omega\times V$ such that 
\begin{eqnarray}\label{3.10}
Q^+(g,g)(y,v)\le c_Q\cdot\|g(y,\cdot)\|_{L^1(V)}
\end{eqnarray} 
for a.e. $v\in V$ whenever $\|g(y,\cdot)\|_{L^1(V)}<\infty$. 
\end{lemma}  
\begin{proof} Let $S'({\mathbb R}^d)$ be the space of all tempered distributions 
on ${\mathbb R}^d$ and let $\hat{f}$ denote the Fourier transform of $f\in S'({ 
\mathbb R}^d)$. If $f\in L^1({\mathbb R}^d)$, it is given by $\hat{f}(\xi)= 
\int_{{\mathbb R}^d}e^{-ix\circ\xi}f(x)\, dx$. For $s\in {\mathbb R}$ introduce 
the Bessel potential spaces 
\begin{eqnarray*} 
H^s({\mathbb R}^d):=\left\{f\in S'({\mathbb R}^d):\hat{f}\in L^2_{\rm loc} 
({\mathbb R}^d),\ \int_{{\mathbb R}^d}(1+|\xi|^2)^s\left|\hat{f}(\xi)\right 
|^2\, d\xi<\infty\right\}\, . 
\end{eqnarray*} 
First we aim to show that for any $k\in {\mathbb N}$ there is a function 
$g^k\in H^{-2+\frac{k(d-1)}{2}}({\mathbb R}^d)$ independent of $r\in\overline 
{\Omega}$ such that $g_\gamma(r,v)\le g^k(v)$ for $(r,v)\in\overline{\Omega} 
\times V$. 
\medskip 

For a.e. $(r,v)\in\overline{\Omega}\times V$ we have by (\ref{3.6}) and $\psi_g 
\ge 1$
\begin{align}\label{3.11}
\begin{split} 
&\hspace{-.5cm}g_\gamma(r,v)=\int_{y\in\Omega}g(y,v)\, h_\gamma(r,y)\, dy \\ 
&\hspace{.5cm}\le\|h_\gamma\|\int_{y\in\Omega}g(y,v)\, dy \\ 
&\hspace{.5cm}\le\|h_\gamma\|\int_{y\in\Omega}g(y^-(y,v),v)\, dy \\ 
&\hspace{1.0cm}+\|h_\gamma\|\int_{y\in\Omega}\int_0^{T_\Omega(y,v)}\lambda Q^+ 
(g,g)(y-sv,v)\, ds\, dy\, . 
\end{split} 
\end{align}
We shall treat both items on the right-hand side of (\ref{3.11}) individually. 
By the definition of $y^-\equiv y^-(y,v)\in\partial\Omega$ in preparation of 
Lemma \ref{Lemma3.3}, we have $v\circ n(y^-)\le 0$ for all $(y,v)\in\Omega\times 
V$. In fact, $v$ points from $y^-\in\partial\Omega$ to the inner of $\Omega$ 
while $n(y^-)$ is the outer normal at $y^-\in\partial\Omega$. Here, the boundary 
conditions (\ref{2.1}) say that 
\begin{eqnarray*} 
g(y^-,v)=J(y^-,\cdot)(g)\cdot M(y^-,v)\, . 
\end{eqnarray*} 
Next we recall that according to (ii), there exist $M_{\rm min},M_{\rm 
max}\in (0,\infty)$ such that $M_{\rm min}\le M(z,v)\le M_{\rm max}$ for 
all $z\in\partial\Omega$ and all $v\in V$ with $v\circ n(z)\le 0$. We fix 
$v\in V$ for the next chain of equations and inequalities and obtain 
\begin{eqnarray*} 
&&\hspace{-.5cm}\int_{y\in\Omega}g(y^-(y,v),v)\, dy=\int_{y\in\Omega}J(y^- 
(y,v),\cdot)(g)M(y^-(y,v),v)\, dy \\ 
&&\hspace{.5cm}\le M_{\rm max}\int_{y\in\Omega}J(y^-(y,v),\cdot)(g)\, dy \\ 
&&\hspace{.5cm}=M_{\rm max}\int_{\{r\in\partial\Omega:v\circ n(r)\le 0\}} 
\int_{t\in [0,T_\Omega(r,-v)]}J(r,\cdot)(g)|v|\, dt\, \left(\frac{-v}{|v|} 
\right)\circ n(r)\, dr \\ 
&&\hspace{.5cm}=M_{\rm max}\int_{\{r\in\partial\Omega:v\circ n(r)\le 0\}} 
\int_{t\in [0,T_\Omega(r,-v/|v|)]}J(r,\cdot)(g)\, dt\, \frac{|v\circ n(r)|} 
{|v|}\, dr \\ 
&&\hspace{.5cm}=M_{\rm max}\int_{\{r\in\partial\Omega:v\circ n(r)\le 0\}} 
\, J(r,\cdot)(g)\cdot |r^+(r,v)-r|\, \frac{|v\circ n(r)|}{|v|}\, dr \\ 
&&\hspace{.5cm}\le {\rm diam}(\Omega) M_{\rm max}\int_{\{r\in\partial\Omega: 
v\circ n(r)\le 0\}}J(r,\cdot)(g)\, dr  
\end{eqnarray*} 
where in the third line we have put $y^-(y,v)=:r$ which implies that $y=
r+tv$ for some $t\in [0,T_\Omega(r,-v)]$. In particular, this substitution 
yields $dy=|v|\, dt\cdot (-v/|v|)\circ n(r)\, dr$. 

Moreover set 
\begin{eqnarray*}
C_M:=\left(M_{\rm min}\cdot\int_{\{w\in V:w\circ n(r)\le 0\}}\frac{|w 
\circ n(r)|^2}{|w|^2}\, dw\right)^{-1} 
\end{eqnarray*} 
and note that $C_M\in (0,\infty)$ is independent of $r\in\partial \Omega$. 
Taking into consideration (\ref{2.1}) we verify that 
\begin{eqnarray*} 
&&\hspace{-.5cm}J(r,\cdot)(g)\le C_MJ(r,\cdot)(g)\int_{\{w\in V:w\circ n(r) 
\le 0\}}\frac{|w\circ n(r)|^2}{|w|^2}M(r,w)\, dw \\ 
&&\hspace{.5cm}=C_M\int_{\{w\in V:w\circ n(r)\le 0\}}\frac{|w\circ n(r)|^2} 
{|w|^2}g(r,w)\, dw 
\end{eqnarray*} 
and thus 
\begin{align}\label{3.12}
\begin{split} 
&\hspace{-.5cm}\int_{y\in\Omega}g(y^-(y,v),v)\, dy \\ 
&\hspace{.5cm}\le {\rm diam}(\Omega)\cdot C_MM_{\rm max}\int_{r\in\partial 
\Omega}\int_{\{w\in V:w\circ n(r)\le 0\}}\frac{|w\circ n(r)|^2}{|w|^2}g(r,w) 
\, dw\, dr \\ 
&\hspace{.5cm}={\rm diam}(\Omega)\cdot C_MM_{\rm max}\int_{w\in V}\int_{ 
\{r\in\partial\Omega:w\circ n(r)\le 0\}}\frac{|w\circ n(r)|^2}{|w|^2}g(r,w) 
\, dr\, dw \, . 
\end{split} 
\end{align} 
if the right-hand side is finite. Keeping in mind that $\Omega$ is a bounded 
domain with smooth boundary it turns out that there is a constant $C_\Omega>0$ 
only depending on $\Omega$ such that 
\begin{eqnarray}\label{3.13}
\left|\frac{r-y}{|r-y|}\circ n(r)\right|\le C_\Omega|r-y|\quad y,r\in\partial 
\Omega. 
\end{eqnarray} 
As a consequence, we have $|w\circ n(r)|/|w|\le C_\Omega\, |r^+(r,w)-r|$ for 
all $w\in V$ and all $r\in\partial\Omega$ with $w\circ n(r)\le 0$. Thus, for 
any $w\in V$ and $r\in\partial\Omega$ such that $w\circ n(r)\le 0$ it holds that 
\begin{eqnarray*}
\frac{|w\circ n(r)|}{|w|}\, g(r,w)\le C_\Omega\int_{t=0}^{T_\Omega(r^+,w)} 
g(r,w)|w|\, dt\, .
\end{eqnarray*} 
Since $g(r,w)=g((r+tw)-tw,w)\le\sup\psi_g\, g(r+tw,w)$ for $t\in [0,T_\Omega 
(r^+,w)]$ by (\ref{3.6}), we have 
\begin{eqnarray}\label{3.14}
\frac{|w\circ n(r)|}{|w|}\, g(r,w)\le C_\Omega\sup\psi_g\int_{t=0}^{T_\Omega 
(r^+,w)}g(r+tw,w)|w|\, dt\, .
\end{eqnarray} 
This and 
\begin{eqnarray*}
&&\hspace{-.5cm}\int_{w\in V}\int_{\{r\in\partial\Omega:w\circ n(r)\le 0\}} 
\frac{|w\circ n(r)|}{|w|}\int_{t=0}^{T_\Omega(r^+,w)}g(r+tw,w)|w|\, dt\, dr 
\, dw \\ 
&&\hspace{.5cm}=\|g\|_{L^1(\Omega\times V)}=1\vphantom{\int}
\end{eqnarray*} 
imply that the right-hand side of (\ref{3.12}) is finite. Writing $\t C$ for 
$C_\Omega\cdot{\rm diam}(\Omega)\cdot\sup\psi_g\cdot C_M M_{\rm max}$ we 
deduce from (\ref{3.12}), (\ref{3.14}), and the last calculation that 
\begin{eqnarray}\label{3.15}
\int_{y\in\Omega}g(y^-(y,v),v)\, dy\le\t C\, . 
\end{eqnarray} 

In order to find an upper bound for the second item of (\ref{3.11}) we shall 
apply the main result of \cite{Li94}, namely Theorem IV.1 and Remark ii). 
Note also the reformulation in Theorem L of \cite{Lu98}. In this regard let 
us recall the particular form of $B$, $B(v,v_1,e)={\bf B}(|v-v_1|,\theta)$ for 
some ${\bf B}\in C_c^\infty((0,\infty)\times (0,\pi/2))$ where $\theta=\arccos 
\left(e\circ (v-v_1)/|v-v_1|\right)\in [0,\pi/2)$. 

By the first line of (\ref{3.11}) we have 
\begin{eqnarray}\label{3.16}
g^0:=\sup_{r\in\overline{\Omega}}g_\gamma(r,\cdot)\in L^1({\mathbb R}^d)\subseteq 
\{\hat{f}:f\in L^\infty({\mathbb R}^d)\}\subseteq H^{-2}({\mathbb R}^d) 
\end{eqnarray} 
For the last inclusion consult \cite{Ho03}, Theorem 7.9.3. Moreover,  
\begin{eqnarray}\label{3.17} 
\int_\Omega g(y,\cdot)\, dy\in L^1({\mathbb R}^d)\, . 
\end{eqnarray} 
Now extend $B$ to ${\mathbb R}^d\times {\mathbb R}^d$ by setting zero outside of 
$V\times V$ and introduce
\begin{eqnarray}\label{3.18}
\t g^1(y,v):=\lambda\int_{{\mathbb R}^d}\int_{S_+^{d-1}}B(v,v_1,e)\, g(y,v^\ast) 
g^0(v_1^\ast)\, de\, dv_1\, , 
\end{eqnarray} 
$y\in\Omega$, $v\in {\mathbb R}^d$. By (\ref{3.3}) we obtain   
\begin{align}\label{3.19}
\begin{split}
&\hspace{-.5cm}\int_{y\in\Omega}\int_0^{T_\Omega(y,v)}\lambda Q^+(g,g)(y-sv,v) 
\, ds\, dy \\ 
&\hspace{.5cm}=\frac{1}{|v|}\int_{y\in\Omega}\int_0^{T_\Omega(y,v)}\lambda Q^+ 
(g,g)(y-sv,v)|v|\, ds\, dy \\ 
&\hspace{.5cm}\le \frac{1}{|v|}\int_{y\in\Omega}\int_0^{T_\Omega(y,v)}\t g^1(y 
-sv,v)|v|\, ds\, dy \\ 
&\hspace{.5cm}\le\frac{{\rm diam}(\Omega)}{v_{min}}\int_{y\in\Omega}\t g^1(y,v) 
\, dy \\ 
&\hspace{.5cm}=\frac{\lambda\, {\rm diam}(\Omega)}{v_{min}}\int_{{\mathbb R}^d} 
\int_{S_+^{d-1}}B(v,v_1,e)\, \int_\Omega g(y,\displaystyle v^\ast)\, dy\, g^0 
(v_1^\ast)\, de\, dv_1\, . 
\end{split} 
\end{align}  
Keeping (\ref{3.16}) and (\ref{3.17}) in mind, by the just mentioned result 
of J.-L. Lions \cite{Li94}, Theorem IV.1 and the subsequent Remark ii), we may 
claim that
\begin{eqnarray}\label{3.20}
\int_{y\in\Omega}\t g^1(y,\cdot)\, dy\in H^{-2+\frac{d-1}{2}}({\mathbb R}^d)\, . 
\end{eqnarray} 
In addition, using the integration by substitution of Remark \ref{Remark2.1} and 
$B(v^\ast,v_1^\ast,e)=B(v,v_1,e)$ for all $v,v_1\in V$ and $e\in S^{d-1}_+$ for 
which $(v^\ast ,v_1^\ast)\in V\times V$, we find  
\begin{eqnarray*}
&&\hspace{-.5cm}\int_{v\in {\mathbb R}^d}\int_{y\in\Omega}\t g^1(y,v)\, dy\, dv \\ 
&&\hspace{.5cm}=\lambda\int_{{\mathbb R}^d}\int_{{\mathbb R}^d}\int_{S_+^{d-1}} 
B(v,v_1,e)\,\int_\Omega g(y,v)\, dy\, g^0(v_1)\, de\, dv_1\, dv \\ 
&&\hspace{.5cm}\le\lambda\|h_\gamma\|\|B\|<\infty\, .\vphantom{\int}
\end{eqnarray*} 
Thus (\ref{3.20}) defines an element 
\begin{eqnarray}\label{3.21}
\int_{y\in\Omega}\t g^1(y,\cdot)\, dy\in H^{-2+\frac{d-1}{2}}({\mathbb R}^d) 
\cap L^1({\mathbb R}^d)\, .  
\end{eqnarray} 

Let $C_c^\infty({\mathbb R}^d)$ denote the set of all infinitely differentiable 
real functions on ${\mathbb R}^d$ which have compact support. Set $C_1:=\| 
h_\gamma\|\cdot\t C$ and choose a non-negative $\Phi_1\in C_c^\infty ({\mathbb 
R}^d)$ with $C_1\le\Phi_1$ on $V$. With 
\begin{eqnarray*}
c_1:=\|h_\gamma\|\, {\rm diam}(\Omega)\, v_{min}^{-1} 
\end{eqnarray*} 
it follows from (\ref{3.11}), (\ref{3.15}), (\ref{3.19}), and (\ref{3.21}) that 
\begin{eqnarray*}
&&\hspace{-.5cm}g_\gamma(r,\cdot)\le\sup_{r\in \overline{\Omega}}g_\gamma(r 
,\cdot)=g^0 \\ 
&&\hspace{.5cm}\le\Phi_1+c_1\int_{y\in\Omega}\t g^1(y,\cdot)\, dy=:g^1\quad 
\mbox{\rm on $V$ and hence on ${\mathbb R}^d$}
\end{eqnarray*} 
for a.e. $r\in\overline{\Omega}$ and 
\begin{eqnarray*}
g^1\in H^{-2+\frac{d-1}{2}}({\mathbb R}^d)\cap L^1({\mathbb R}^d)\, . 
\end{eqnarray*} 
We re-iterate from (\ref{3.16}) to here replacing $g^0$ by $g^{k-1}$ and $\t 
g^1(y,v)$ by 
\begin{eqnarray}\label{3.22} 
\lambda\int_{{\mathbb R}^d}\int_{S_+^{d-1}}B(v,v_1,e)\, g(y,v^\ast)g^{k-1} 
(v_1^\ast)\, de\, dv_1=:\t g^k(y,v)\, ,
\end{eqnarray} 
$y\in\Omega$, $v\in {\mathbb R}^d$ in order to obtain 
\begin{align}\label{3.23}
\begin{split}
&\hspace{-.5cm}g_\gamma(r,\cdot)\le g^0\le g^1\le\ldots\le g^k \\ 
&\hspace{.5cm}:=\Phi_1+c_1\int_{y\in\Omega}\t g^k(y,\cdot)\, dy \in H^{-2+ 
\frac{k(d-1)}{2}}({\mathbb R}^d)\cap L^1({\mathbb R}^d)
\end{split} 
\end{align} 
for a.e. $r\in\overline{\Omega}$ and for all $k\in {\mathbb N}$. By Lions' theorem 
it also holds that 
\begin{eqnarray}\label{3.24}
\t g^1(y,\cdot)\le\ldots\le\t g^{k+1}(y,\cdot)\in H^{-2+\frac{(k+1)(d-1)}{2}} 
({\mathbb R}^d)  
\end{eqnarray} 
provided that $\|g(y,\cdot)\|_{L^1(V)}<\infty$. 
\medskip 

We mention furthermore that, for those $l=-2+k(d-1)/2$ which are non-negative 
integers, the Bessel potential space $H^{-2+\frac{k(d-1)}{2}}({\mathbb R}^d)$ 
coincides with the Sobolev space $W^{l,2}({\mathbb R}^d)$. In particular the norms 
are equivalent. In this case 
\begin{eqnarray*}
W^{l,2}({\mathbb R}^d)\subseteq L^\infty({\mathbb R}^d)\quad\mbox{\rm 
continuously if}\quad l>\frac{d}{2}\, , 
\end{eqnarray*} 
cf. \cite{Br11}, Corollary 9.13. Relation (\ref{3.23}) implies the a.e. boundedness 
of $g_\gamma$ on $\overline{\Omega}\times V$. We have proved the first statement 
of the lemma. 
\medskip

In order to show (\ref{3.10}) we use the norm estimate in \cite{Li94}, Remark 
ii) to Theorem IV.1. Let $0<C<\infty$ be the constant introduced there. Moreover, 
let $0<C_{l'}<\infty$ denote the constant from the Sobolev inequality between 
the spaces $H^{l'}({\mathbb R}^d)=W^{l',2}({\mathbb R}^d)$ and $L^\infty({\mathbb 
R}^d)$. If $\|g(y,\cdot)\|_{L^1(V)}<\infty$ then according to (\ref{3.18}) and 
(\ref{3.22})-(\ref{3.24}) the following holds. If $l':=-2+(k+1)(d-1)/2$ is a natural 
number and $l'>d/2$ then we have 
\begin{eqnarray*}
&&\hspace{-.5cm}\lambda Q^+(g,g)(y,v)\le\t g^1(y,v)\le\t g^{k+1}(y,v)\le\| 
\t g^{k+1}(y,\cdot)\|_{L^\infty({\mathbb R}^d)} \\ 
&&\hspace{.5cm}\le C_{l'}\|\t g^{k+1}(y,\cdot)\|_{H^{l'}({\mathbb R}^d)} \\ 
&&\hspace{.5cm}\le C_{l'}\, C\cdot\|g^k\|_{H^l({\mathbb R}^d)}\, \|g(y, 
\cdot)\|_{L^1(V)}\equiv\lambda c_Q\|g(y,\cdot)\|_{L^1(V)} 
\end{eqnarray*} 
for a.e. $v\in V$. 
\end{proof}

\section{Boundedness Properties}\label{sec:4}
\setcounter{equation}{0} 

In the proof of the subsequent lemma we shall use the notion of {\it narrow 
convergence} of finite measures on ${\mathbb R}^d$ in the sense of \cite{DMRT13}. 
We say that a sequence of finite measures $\mu_n$ on $({\mathbb R}^d,{\cal B} 
({\mathbb R}^d))$ converges {\it narrowly} to some finite measures $\mu$ on 
$({\mathbb R}^d,{\cal B}({\mathbb R}^d))$ if 
\begin{eqnarray*}
\lim_{n\to\infty}\int f\, d\mu_n=\int f\, d\mu
\end{eqnarray*}
for all real bounded continuous functions $f$ on ${\mathbb R}^d$. 
\begin{lemma}\label{Lemma4.1} 
Let $g$ satisfy (\ref{3.2}) in the sense of Remark \ref{Remark3.2} and suppose 
(\ref{3.1}). The function $\Omega\ni y\mapsto\|g(y,\cdot) \|_{L^1(V)}$ belongs 
to $L^q(\Omega)$ for all $1\le q<\infty$ and that $\partial\Omega\ni r\mapsto 
J(r,\cdot)(g)$ belongs to $L^q(\partial\Omega)$ for all $1\le q<\infty$.
\end{lemma}
\begin{proof} {\it Step 1 } We establish an iteration to prove the lemma. The 
present Step 1 is dedicated to the particular case $q=1$, the {\it initialization 
step} of the iteration. Because of $\|g\|_{L^1(\Omega\times V)}=1$ we just have 
to focus on $\partial\Omega\ni r\mapsto J(r,\cdot)(g)$. 
\medskip 

According to (\ref{2.1}) and (\ref{3.6}) we have 
\begin{align}\label{4.1}
\begin{split}
&\hspace{-.5cm}1=\int_\Omega\int_V g(y,v)\, dv\, dy \\ 
&\hspace{.5cm}\ge\frac{1}{\sup\psi_g}\int_\Omega\int_V g(y^-(y,v),v)\, dv 
\, dy \\ 
&\hspace{.5cm}\ge\frac{M_{\rm min}}{\sup\psi_g}\int_\Omega\int_VJ(y^-(y,v), 
\cdot)(g)\, dv\, dy\, . 
\end{split} 
\end{align}  
We let $v=\alpha e$ where $\alpha\in(v_{min},v_{max})$ and $e\in S^{d-1}$. 
Furthermore we denote by $l_S$ the Riemann-Lebesgue measure on $(S^{d-1}, 
{\cal B}(S^{d-1}))$.  Set 
\begin{eqnarray*}
I_V(m):=\int_{v_{min}}^{v_{max}}\alpha^{m-1}\, d\alpha\, ,\quad m\in {\mathbb 
N}. 
\end{eqnarray*}
It follows that 
\begin{eqnarray*} 
&&\hspace{-.5cm}\int_VJ(y^-(y,v),\cdot)(g)\, dv \\ 
&&\hspace{.5cm}=\int_{v_{min}}^{v_{max}}\alpha^{d-1}\int_{S^{d-1}}J(y^-(y, 
\alpha e),\cdot)(g)\, dl_S(e)\, d\alpha \\ 
&&\hspace{.5cm}\ge I_V(d)\int_{S^{d-1}}J(y^-(y,\cdot\, e), \cdot)(g)\, dl_S(e) 
\end{eqnarray*}
where we note that $y^-(y,\alpha e)\in\partial\Omega$ is independent of $\alpha 
\in (v_{min},v_{max})$ and therefore appears as $y^-(y,\cdot\, e)$ in the second 
line. 
Let us denote 
\begin{eqnarray*} 
(\partial\Omega)_y:=\{r\in\partial\Omega:\{r+a(y-r):a\in (0,1)\}\subset 
\Omega\}\, ,\quad y\in\overline{\Omega}, 
\end{eqnarray*}
and 
\begin{eqnarray*} 
(\Omega)_r:=\{y\in\Omega:\{r+a(y-r):a\in (0,1)\}\subset\Omega\}\, ,\quad 
r\in\partial\Omega\, .    
\end{eqnarray*}
We mention that for $y\in\Omega$ and $r\in\partial\Omega$ we have $r\in 
(\partial\Omega)_y$ if and only if $y\in(\Omega)_r$. Observe also that for 
$e\in S^{d-1}$, $y\in\Omega$, and $r:=y^-(y,\cdot\, e)$ we have $r\in 
(\partial\Omega)_y$, $e=(y-r)/|y-r|$, and 
\begin{eqnarray*} 
dl_S(e)=|y-r|^{1-d}\, dr\cdot n(r)\circ (-e)=|y-r|^{1-d}\cdot\frac{n(r)\circ 
(r-y)}{|y-r|}\, dr\, . 
\end{eqnarray*}
Summarizing the preparations from relation (\ref{4.1}) to here, we obtain 
\begin{align}\label{4.2}
\begin{split}
&\hspace{-.5cm}1\ge\frac{M_{\rm min}\cdot I_V(d)}{\sup\psi_g}\int_\Omega 
\int_{S^{d-1}}J(y^-(y,\cdot\, e), \cdot)(g)\, dl_S(e)\, dy \\ 
&\hspace{.5cm}=\frac{M_{\rm min}\cdot I_V(d)}{\sup\psi_g}\int_{y\in\Omega} 
\int_{r\in(\partial\Omega)_y}|y-r|^{1-d}\cdot\frac{n(r)\circ (r-y)}{|y-r|}J(r, 
\cdot)(g)\, dr\, dy \\ 
&\hspace{.5cm}=\frac{M_{\rm min}\cdot I_V(d)}{\sup\psi_g}\int_{r\in\partial 
\Omega}\int_{y\in(\Omega)_r}|y-r|^{1-d}\cdot\frac{n(r)\circ (r-y)}{|y-r|}J(r, 
\cdot)(g)\, dy\, dr \\ 
&\hspace{.5cm}=C_d\int_{\partial\Omega}J(r,\cdot)(g)\rho_d(r)\, dr\, , 
\end{split} 
\end{align} 
where $C_d:=M_{\rm min}\cdot I_V(d)/\sup\psi_g$ and 
\begin{eqnarray*} 
\rho_d(r):=\int_{(\Omega)_r}|y-r|^{1-d}\cdot\frac{n(r)\circ (r-y)}{|y-r|}\, dy\, . 
\end{eqnarray*}
There is some $c_d>0$, which depends on $\Omega$ but not on $r\in\partial\Omega$, 
such that $\infty>\rho_d(r)\ge c_d$ for all $r\in\partial\Omega$. Thus from 
(\ref{4.2}) we may conclude that 
\begin{eqnarray}\label{4.3}
\int_{\partial\Omega}J(r,\cdot)(g)\, dr\le\frac{1}{c_d\, C_d}\, . 
\end{eqnarray}

\noindent 
{\it Step 2 } In this step we assume that $\Omega\ni y\mapsto\|g(y,\cdot) \|_{L^1 
(V)}$ belongs to $L^q(\Omega)$ and that $\partial\Omega\ni r\mapsto J(r,\cdot)(g)$ 
belongs to $L^q(\partial\Omega)$ for some $1\le q<\infty$. It is our aim to show 
that $\Omega\ni y\mapsto\|g(y,\cdot)\|_{L^1(V)}$ belongs to $L^{pq}(\Omega)$ and 
that $\partial\Omega\ni r\mapsto J(r,\cdot)(g)$ belongs to $L^{pq}(\partial\Omega)$ 
for all $1\le p<d/(d-1)$. In other words, the present Step 2 is the {\it execution 
step} of the iteration.
\medskip 

According to (\ref{3.6}) it holds with an appropriate constant $0<c_q<\infty$ that 
\begin{align}\label{4.4}
\begin{split}
&\hspace{-.5cm}\|g(y,\cdot)\|^q_{L^1(V)}\le c_q\int_V (g(y^-(y,v),v))^q\, dv 
 \\ 
&\hspace{1.0cm}+c_q\left(\int_V\int_0^{T_\Omega}\lambda Q^+(g,g)(y-sv,v)\, ds 
\, dv\right)^q\quad\mbox{\rm for a.e. }y\in \Omega 
\end{split} 
\end{align}
where we have taken into consideration that $\psi_g\ge 1$. The two integrals on 
the right-hand side are now treated separately. For this we let $v=\alpha e$ where 
$\alpha\in(v_{min},v_{max})$ and $e\in S^{d-1}$. We recall that $l_S$ denotes the 
Riemann-Lebesgue measure on $(S^{d-1},{\cal B}(S^{d-1}))$. We have 
\begin{eqnarray*} 
&&\hspace{-.5cm}\int_V (g(y^-(y,v),v))^q\, dv \\ 
&&\hspace{.5cm}=\int_{v_{min}}^{v_{max}}\alpha^{d-1}\int_{S^{d-1}}(J(y^-(y,\alpha e), 
\cdot)(g))^q (M(y^-,\alpha e))^q\, dl_S(e)\, d\alpha \\ 
&&\hspace{.5cm}\le M_{\rm max}^qI_V(d)\int_{S^{d-1}}(J(y^-(y,\cdot\, e),\cdot)(g))^q 
\, dl_S(e) 
\end{eqnarray*}
where we note again that $y^-(y,\alpha e)\in\partial\Omega$ is independent of 
$\alpha\in(v_{min},v_{max})$ and therefore appears as $y^-(y,\cdot\, e)$ in the 
last line. We obtain 
\begin{align}\label{4.5}
\begin{split} 
&\hspace{-.5cm}\int_V (g(y^-(y,v),v))^q\, dv \\ 
&\hspace{.5cm}\le M_{\rm max}^qI_V(d)\int_{r\in (\partial\Omega)_y}|y-r|^{1-d}\cdot 
\frac{n(r)\circ (r-y)}{|y-r|}(J(r,\cdot)(g))^q\, dr \\ 
&\hspace{.5cm}\le M_{\rm max}^qI_V(d)\int_{r\in\partial\Omega}|y-r|^{1-d}\, (J(r, 
\cdot)(g))^q\, dr\, . 
\end{split} 
\end{align}

Introduce
\begin{eqnarray*} 
f(s):=\left\{ 
\begin{array}{cl}
\D\exp\left\{-1/\left(1-s^2\vphantom{l^1}\right)\vphantom{\dot{f}}\right\}\, , 
& s\in\left(-1,1\right), \\ 
0\, , & s\in {\mathbb R}\setminus\left(-1,1\right),\vphantom{\D\dot{f}} 
\end{array} 
\right. , 
\end{eqnarray*} 
$m_d:=\left(\int_{{\mathbb R}^d} f(|x|)\, dx\right)^{-1}$, and the mollifier 
function $\psi_d(x):=m_d\cdot f(|x|)$, $x\in {\mathbb R}^d$. Define 
\begin{eqnarray*} 
\gamma_n(x):=n^d\int_{r\in\partial\Omega}\psi_d(n(x-r))(J(r,\cdot)(g))^q\, dr 
\, ,\quad x\in {\mathbb R}^d,\ n\in {\mathbb N}.  
\end{eqnarray*}
We observe that the sequence of measures $\Gamma_n$, $n\in {\mathbb N}$, given 
by $\Gamma_n(A):=\int_A\gamma_n(x)\, dx$, $A\in {\cal B}({\mathbb R}^d)$, converges 
narrowly as $n\to\infty$ to the measure $\Gamma$ defined by 
\begin{eqnarray*}
\Gamma(A):=\int_{A\cap\partial\Omega}(J(r,\cdot)(g))^q\, dr\, ,\quad A\in {\cal 
B}({\mathbb R}^d). 
\end{eqnarray*}
We remark that the measure $\Gamma$ is finite because of the assumption of 
Step 2 that $\partial\Omega\ni r\mapsto J(r,\cdot)(g)$ belongs to $L^q(\partial 
\Omega)$. 

Introduce $\Omega_1:=\{z\in {\mathbb R}^d:\inf_{y\in\Omega}|z-y|<1\}$. 
The inequality (\ref{4.5}) implies now  
\begin{align}\label{4.6}
\begin{split} 
&\hspace{-.5cm}\int_V (g(y^-(y,v),v))^q\, dv \\ 
&\hspace{.5cm}\le M_{\rm max}^qI_V(d)\int_{r\in\partial\Omega}|y-r|^{1-d}\, 
(J(r,\cdot)(g))^q\, dr \\ 
&\hspace{.5cm}=M_{\rm max}^qI_V(d)\, \lim_{n\to\infty}\int_{x\in\Omega_1}|y-x 
|^{1-d}\gamma_n(x)\, dx 
\end{split} 
\end{align}
where the sequence of functions ${\mathbb R}^d\ni y\mapsto\int_{x\in\Omega_1} 
|y-x|^{1-d}\gamma_n(x)\, dx$, $n\in {\mathbb N}$, is bounded in $L^p({\mathbb 
R}^d)$ for $1\le p<d/(d-1)$ because of 
\begin{align}\label{4.7}
\begin{split}
&\hspace{-.5cm}\sup_{n\in {\mathbb N}}\int_\Omega\left(\int_{x\in\Omega_1} 
|y-x|^{1-d}\gamma_n(x)\, dx\right)^p\, dy \\ 
&\hspace{.5cm}\le\sup_{n\in {\mathbb N}}\int_\Omega\left(\int_{x\in\Omega_1} 
\gamma_n(x)\, dx\right)^{p-1}\int_{x\in\Omega_1}|y-x|^{(1-d)p}\gamma_n(x)\, 
dx\, dy \\ 
&\hspace{.5cm}\le\sup_{n\in {\mathbb N}}\left(\int_{x\in\Omega_1}\gamma_n(x) 
\, dx\right)^p\sup_{x\in\Omega_1}\int_\Omega|y-x|^{(1-d)p}\, dy\, , 
\end{split} 
\end{align}
the definition of $\gamma_n$, $n\in {\mathbb N}$, and the finiteness of the 
measure $\Gamma$. 
\medskip

Furthermore, the limit in (\ref{4.6}) is weakly in $L^p(\Omega)$ for $1\le p< 
d/(d-1)$, i. e. $\Omega\ni y\mapsto\int_{r\in\partial\Omega}|y-r|^{1-d}\, (J(r, 
\cdot)(g))^q\, dr$, belongs also to $L^p(\Omega)$ for $1\le p<d/(d-1)$ under the 
assumption of Step 2 that $\int_{r\in\partial\Omega}(J(r,\cdot)(g))^q\, dr< 
\infty$. The mode of convergence in (\ref{4.6}), i. e. weak convergence in 
$L^p(\Omega)$, can be justified as follows. On the one hand, for any test 
function $\vp\in C_b(\Omega)$ and 
\begin{eqnarray*}
\Phi(x):=\int_{y\in{\Omega}}\vp(y)\, |y-x|^{(1-d)}\, dy\, ,\quad x\in{\mathbb R}^d, 
\end{eqnarray*}
we have $\Phi\in C_b({\mathbb R}^d)$ and 
\begin{eqnarray*} 
&&\hspace{-.5cm}\int_{y\in\Omega}\vp(y)\int_{x\in\Omega_1}|y-x|^{1-d}\gamma_n 
(x)\, dx\, dy \\ 
&&\hspace{.5cm}=\int_{x\in\Omega_1}\Phi(x)\, \gamma_n(x)\, dx\stack{n\to\infty} 
{\lra}\int_{r\in\partial\Omega}\Phi(r)(J(r,\cdot)(g))^q\, dr \\ 
&&\hspace{.5cm}=\int_{y\in\Omega}\vp(y)\int_{r\in\partial\Omega}|y-r|^{1-d}\, 
(J(r,\cdot)(g))^q\, dr\, dy
\end{eqnarray*} 
by the narrow convergence of $\Gamma_n$ to $\Gamma$ as $n\to\infty$; note that 
$\Gamma_n({\mathbb R}^d\setminus\Omega_1)=\Gamma({\mathbb R}^d\setminus\partial 
\Omega)=0$. On the other hand, by (\ref{4.7}), and the assumption $\int_{r\in 
\partial\Omega}(J(r,\cdot)(g))^q\, dr<\infty$, 
\begin{eqnarray*} 
&&\hspace{-.5cm}\limsup_{n\to\infty}\left\|\int_{x\in\Omega_1}|\cdot-x|^{1-d} 
\gamma_n(x)\, dx\right\|^p_{L^p(\Omega)} \\ 
&&\hspace{.5cm}\le\lim_{n\to\infty}\left(\int_{x\in\Omega_1}\gamma_n(x)\, dx
\right)^p\sup_{x\in\Omega_1}\int_\Omega|y-x|^{(1-d)p}\, dy \\ 
&&\hspace{.5cm}=\left(\int_{r\in\partial\Omega}(J(r,\cdot)(g))^q\, dr\right)^p 
\sup_{x\in\Omega_1}\int_\Omega|y-x|^{(1-d)p}\, dy<\infty\, . 
\end{eqnarray*} 
The last two chains of equalities and inequalities verify the claimed weak 
convergence in $L^p(\Omega)$ in (\ref{4.6}). 
\medskip 

Setting 
\begin{eqnarray*} 
C_{Q,v}:=c_qc_Q^q\lambda^q(I_V(d-1))^q(l_S(S^{d-1}))^{q-1}\, {\rm diam}(\Omega 
)^{q-1}
\end{eqnarray*} 
for the second integral in (\ref{4.4}) we obtain from (\ref{3.10})
\begin{align}\label{4.8}
\begin{split}  
&\hspace{-.5cm}\left(\int_V\int_0^{T_\Omega(y,v)}\lambda Q^+(g,g)(y-sv,v)\, ds 
\, dv\right)^q \\ 
&\hspace{.5cm}=\left(\int_{v_{min}}^{v_{max}}\alpha^{d-1}\int_{S^{d-1}}\int_0^{ 
T_\Omega(y,\alpha e)}\lambda Q^+(g,g)(y-s\, \alpha e,\alpha e)\, ds\, dl_S(e)\, 
d\alpha\right)^q \\ 
&\hspace{.5cm}\le c_Q^q\lambda^q\left(\int_{v_{min}}^{v_{max}}\hspace{-.3cm} 
\alpha^{d-2}\int_{S^{d-1}}\int_0^{T_\Omega(y,\alpha e)}\hspace{-.3cm}\|g(y-s\, 
\alpha e,\cdot)\|_{L^1(V)}\, \alpha\, ds\, dl_S(e)\, d\alpha\right)^q \\ 
&\hspace{.5cm}\le c_Q^q\lambda^q(I_V(d-1))^q\left(\int_{S^{d-1}}\int_0^{T_\Omega 
(y,e)}\|g(y-se,\cdot)\|_{L^1(V)}\, ds\, dl_S(e)\right)^q \\ 
&\hspace{.5cm}\le\frac{C_{Q,v}}{c_q}\int_{S^{d-1}}\int_0^{T_\Omega(y,e)}\|g(y-se, 
\cdot)\|^q_{L^1(V)}\, ds\, dl_S(e) \\ 
&\hspace{.5cm}\le\frac{C_{Q,v}}{c_q}\int_\Omega |x-y|^{1-d}\cdot\|g(x,\cdot)\|^q_{ 
L^1(V)}\, dx\, .
\end{split} 
\end{align}

With 
\begin{eqnarray*}
C_{M,v}:=c_qM_{\rm max}^qI_V(d) 
\end{eqnarray*}
it follows now from (\ref{4.4}), (\ref{4.6}), and (\ref{4.8}) that for a.e. 
$y\in\Omega$ 
\begin{align}\label{4.9}
\begin{split} 
&\hspace{-.5cm}\|g(y,\cdot)\|^q_{L^1(V)}\le C_{M,v}\lim_{n\to\infty}\int_{x 
\in\Omega_1}|y-x|^{1-d}\gamma_n(x)\, dx \\ 
&\hspace{1.0cm}+C_{Q,v}\int_{x\in\Omega}|y-x|^{1-d}\|g(x,\cdot)\|^q_{L^1(V)} 
\, dx\, .  
\end{split} 
\end{align}
According to the argumentation below (\ref{4.7}), 
\begin{eqnarray*}
\Omega\ni y\mapsto\int_{r 
\in\partial\Omega}|y-r|^{1-d}\, (J(r,\cdot)(g))^q\, dr\quad\mbox{\rm belongs 
to }L^p(\Omega) 
\end{eqnarray*}
for $1\le p<d/(d-1)$ by the assumption $\int_{r\in\partial\Omega}(J(r,\cdot)(g) 
)^q\, dr<\infty$. Furthermore, according to 
\begin{eqnarray*} 
&&\hspace{-.5cm}\int_\Omega\left(\int_{x\in\Omega}|y-x|^{1-d}\|g(x,\cdot)\|^q_{ 
L^1(V)}\, dx\right)^p\, dy \\ 
&&\hspace{.5cm}\le\left(\int_{x\in\Omega}\|g(x,\cdot)\|^q_{L^1(V)}\, dx\right)^p 
\sup_{x\in\Omega}\int_\Omega|y-x|^{(1-d)p}\, dy\, , 
\end{eqnarray*}
the function 
\begin{eqnarray*}
\Omega\ni y\mapsto\int_{x\in\Omega}|y-x|^{1-d}\|g(x,\cdot)\|^q_{L^1(V)}\, dx 
\quad\mbox{\rm belongs to }L^p(\Omega)  
\end{eqnarray*}
for $1\le p<d/(d-1)$ under the assumption of Step 2 that $\int_\Omega\|g(x,\cdot) 
\|^q_{L^1(V)}\, dx<\infty$. With this discussion in mind, it follows from 
(\ref{4.9}) that    
\begin{align}\label{4.10}
\begin{split} 
\int_{\Omega}\|g(y,\cdot )\|^{pq}_{L^1(V)}\, dy<\infty\quad\mbox{\rm if} 
\hphantom{aaaaaaaaaaaaa}\\ 
\int_{\partial\Omega}(J(r,\cdot))^q\, dr<\infty\quad \mbox{\rm and}\quad 
\int_{\Omega}\|g(x,\cdot )\|^q_{L^1(V)}\, dx<\infty.
\end{split} 
\end{align}

According to (\ref{3.6}) we have for some suitable constant $0<c_{J,1}<\infty$ 
only depending on $V$, $p$, and $q$  
\begin{align}\label{4.11}
\begin{split}
&\hspace{-.5cm}(J(r,\cdot)(g))^{pq}=\left(\int_{\{v\in V:v\circ n(r)\ge 0\}} 
v\circ n(r)\, g(r,v)\, dv\right)^{pq} \\ 
&\hspace{.5cm}\le c_{J,1}M_{\rm max}^{pq}\left(\int_{\{v\in V:v\circ n(r)\ge 
0\}}v\circ n(r)\, J(r^-,\cdot )(g)\, dv\right)^{pq} \\ 
&\hspace{1.0cm}+c_{J,1}\left(\int_{\{v\in V:v\circ n(r)\ge 0\}}v\circ n(r) 
\int_0^{T_\Omega(r,v)}\lambda Q^+(g,g)(r-sv,v)\, ds\, dv\right)^{pq}
\end{split} 
\end{align}
for a.e. $r\in\partial\Omega$. As above, the two items on the right-hand side 
will be treated separately. Recalling the definition of $C_\Omega $ in (\ref{3.13}) 
we get for the first item of the right-hand side of (\ref{4.11})
\begin{align}\label{4.12}
\begin{split} 
&\hspace{-.5cm}\int_{\{v\in V:v\circ n(r)\ge 0\}}v\circ n(r)\, J(r^-,\cdot ) 
(g)\, dv \\
&\hspace{.5cm}=\int_{v_{min}}^{v_{max}}\alpha^{d-1}\int_{e\in S_+^{d-1}(n(r))} 
\alpha e\circ n(r)\, J(r^-(r,\cdot e),\cdot )(g)\, dl_S(e)\, d\alpha \\ 
&\hspace{.5cm}=\int_{v_{min}}^{v_{max}}\alpha^d\int_{e\in S_+^{d-1}(n(r))}e 
\circ n(r)\, J(r^-(r,\cdot e),\cdot )(g)\, dl_S(e)\, d\alpha \\ 
&\hspace{.5cm}=I_V(d+1)\int_{y\in(\partial\Omega)_r}\frac{(r-y)\circ n(r)} 
{|r-y|}\, J(y,\cdot )(g)|r-y|^{1-d}\cdot\frac{n(y)\circ (y-r)}{|r-y|}\, dy \\ 
&\hspace{.5cm}\le C_\Omega^2 I_V(d+1)\int_{y\in(\partial\Omega)_r}J(y,\cdot)(g) 
|r-y|^{3-d}\, dy \\ 
&\hspace{.5cm}\le C_\Omega^2I_V(d+1)({\rm diam}(\Omega))^{3-d}\int_{y\in\partial 
\Omega}J(y,\cdot)(g)\, dy  
\end{split} 
\end{align}
where we are reminded of $d=2$ or $d=3$. Consequently, with some suitable constant 
$0<c_{J,2}<\infty$ only depending on $\Omega$, $V$, $p$, and $q$ it holds that 
\begin{eqnarray*}
&&\hspace{-.5cm}\left(\int_{\{v\in V:v\circ n(r)\ge 0\}}v\circ n(r)\, J(r^-,\cdot ) 
(g)\, dv\right)^{pq} \\ 
&&\hspace{.5cm}\le c_{J,2}\left(\int_{\{v\in V:v\circ n(r)\ge 0\}}(J(y,\cdot )(g) 
)^q\, dy\right)^p\, . 
\end{eqnarray*}
Recalling the assumption $\int_{y\in\partial\Omega}(J(y,\cdot )(g))^q\, dy< 
\infty$ we obtain 
\begin{eqnarray}\label{4.13} 
\int_{r\in\partial\Omega}\left(\int_{\{v\in V:v\circ n(r)\ge 0\}}v\circ n(r)\, 
J(r^-,\cdot )(g)\, dv\right)^{pq}\, dr<\infty\, . 
\end{eqnarray}

Let us turn to the second item in (\ref{4.11}). Here we follow and slightly 
modify the calculations performed in (\ref{4.8}) to obtain 
\begin{align}\label{4.14}
\begin{split} 
&\hspace{-0.5cm}\int_{\{v\in V:v\circ n(r)\ge 0\}}v\circ n(r)\int_0^{T_\Omega 
(r,v)}\lambda Q^+(g,g)(r-sv,v)\, ds\, dv \\ 
&\hspace{0.5cm}\le c_Q\lambda\int_{\{v\in V:v\circ n(r)\ge 0\}}\int_0^{T_\Omega 
(r,v)}\|g(r-sv,\cdot )\|_{L^1(V)}|v|\, ds\, dv \\ 
&\hspace{0.5cm}=c_Q\lambda\int_{v_{min}}^{v_{max}}\alpha^{d-1}\int_{S_+^{d-1} 
(n(r))}\int_0^{T_\Omega(r,e)}\|g(r-se,\cdot)\|_{L^1(V)}\, ds\, dl_S(e)\, d\alpha 
 \\ 
&\hspace{0.5cm}\le c_Q\lambda I_V(d)\int_{x\in (\Omega)_r}|r-x|^{1-d}\cdot\| 
g(x,\cdot )\|_{L^1(V)}\, dx\, .  
\end{split} 
\end{align} 
Noting that $\int_{x\in\Omega}\|g(x,\cdot)\|_{L^1(V)}\, dx=1$ as well as 
\begin{eqnarray*}
c_3:=\sup_{y\in{\partial\Omega}}\int_{\Omega}|y-x|^{(1-d)p}\, dx<\infty\quad 
\mbox{\rm for }1\le p<d/(d-1) 
\end{eqnarray*}

and setting 
\begin{eqnarray*}
c_{J,3}:=c_3^{q-1}(c_Q\lambda I_V(d))^{pq} 
\end{eqnarray*}
we continue by 
\begin{align}\label{4.15}
\begin{split}
&\hspace{-0.5cm}\left(\int_{\{v\in V:v\circ n(r)\ge 0\}}v\circ n(r)\int_0^{ 
T_\Omega(r,v)}\lambda Q^+(g,g)(r-sv,v)\, ds\, dv\right)^{pq} \\  
&\hspace{0.5cm}\le (c_Q\lambda I_V(d))^{pq}\left(\int_{\Omega}|r-x|^{(1-d)p} 
\|g(x,\cdot)\|_{L^1(V)}\, dx\right)^q \\  
&\hspace{0.5cm}\le c_{J,3}\int_{\Omega}|r-x|^{(1-d)p}\|g(x,\cdot )\|^q_{L^1 
(V)}\, dx\, . 
\end{split} 
\end{align} 
Let us recall the definition $\Omega_1=\{z\in {\mathbb R}^d:\inf_{y\in\Omega} 
|z-y|<1\}$. Moreover, let us take into consideration 
\begin{eqnarray*}
\int_{x\in {\mathbb R}^d}|r-x|^{(1-d)p}\cdot |x-z|^{1-d}\, dx=c_{p,d}\cdot 
|r-z|^{1+(1-d)p}\, ,\quad z\in {\mathbb R}^d, 
\end{eqnarray*}
for some constant $0<c_{p,d}<\infty$ only depending on $p$ and $d$. Furthermore 
observe that $1+(1-d)p>1-d$ for $1\le p<d/(d-1)$ and thus 
\begin{eqnarray}\label{4.16} 
\sup_{z\in\Omega_1}\int_{\partial\Omega}|r-z|^{1+(1-d)p}\, dr<\infty\, ,\quad 1 
\le p<d/(d-1).  
\end{eqnarray}
According to (\ref{4.9}) we have 
\begin{eqnarray*}
&&\hspace{-.5cm}\int_\Omega |r-x|^{(1-d)p}\|g(x,\cdot )\|^q_{L^1(V)}\, dx\\
&&\hspace{.5cm}\le C_{M,v}\int_\Omega |r-x|^{(1-d)p}\left(\lim_{n\to\infty} 
\int_{\Omega_1}|x-z|^{1-d}\cdot\gamma_n(z)\, dz\right)\, dx \\ 
&&\hspace{1.0cm}+c_{p,d}C_{Q,v}\int_\Omega|r-z|^{1+(1-d)p}\cdot\|g(z,\cdot ) 
\|^q_{L^1(V)}\, dz\, . 
\end{eqnarray*}
We remark that the integration $\int_\Omega |r-x|^{(1-d)p}\ldots\, dx$ is 
the actual reason for the weak limit in $L^p(\Omega)$ in (\ref{4.6}). Now 
we take advantage of the facts that $\Omega\ni x\mapsto|r-x|^{(1-d)p}\in 
L^{p/(p-1)}(\Omega)$ and $\Omega\ni x\mapsto\int_{\Omega_1}|x-z|^{1-d}\cdot 
\gamma_n(z)\, dz$ converges weakly in $L^p(\Omega)$ to $\Omega\ni x\mapsto 
\int_{\partial \Omega}|x-y|^{1-d} (J(y,\cdot))^q\, dy$  for $1\le p<d/(d-1)$. 
For the latter recall the discussion below (\ref{4.7}). We get 
\begin{align}\label{4.17}
\begin{split}  
&\hspace{-.5cm}\int_\Omega |r-x|^{(1-d)p}\|g(x,\cdot )\|^q_{L^1(V)}\, dx \\ 
&\hspace{.5cm}\le C_{M,v}\int_\Omega |r-x|^{(1-d)p}\int_{\partial\Omega} 
|x-y|^{1-d}(J(y,\cdot))^q\, dy\, dx \\ 
&\hspace{1.0cm}+c_{p,d}C_{Q,v}\int_\Omega|r-z|^{1+(1-d)p}\cdot\|g(z,\cdot ) 
\|^q_{L^1(V)}\, dz \\ 
&\hspace{.5cm}\le c_{p,d}C_{M,v}\int_{\partial\Omega}|r-y|^{1+(1-d)p}(J(y, 
\cdot))^q\, dy \\ 
&\hspace{1.0cm}+c_{p,d}C_{Q,v}\int_{\Omega}|r-z|^{1+(1-d)p}\|g(z,\cdot ) 
\|^q_{L^1(V)}\, dz\, .
\end{split} 
\end{align}
Because of (\ref{4.16}) the right-hand side, and hence the left-hand side, of 
(\ref{4.17}) is integrable with respect to the variable $r$ and the Lebesgue 
measure on $(\partial\Omega,{\cal B}(\partial\Omega))$ under the assumptions 
of Step 2, $\int_{\partial\Omega}(J(y,\cdot))^q\, dy<\infty$ and $\int_{\Omega} 
\|g(z,\cdot )\|^q_{L^1(V)}\, dz<\infty$. We obtain from (\ref{4.15}), (\ref{4.16}), 
and (\ref{4.17})
\begin{eqnarray}\label{4.18} 
\quad\int_{r\in\partial\Omega}\left(\int_{v\circ n(r)\ge 0}v\circ n(r)\int_0^{ 
T_\Omega (r,v)}\lambda Q^+(g,g)(r-sv,v)\, ds\, dv\right)^{pq}\, dr<\infty
\end{eqnarray}
under the assumptions $\int_{\partial\Omega}(J(y,\cdot))^q\, dy<\infty$ and 
$\int_{\Omega}\|g(z,\cdot )\|^q_{L^1(V)}\, dz<\infty$. Reviewing (\ref{4.11}), 
(\ref{4.13}), and (\ref{4.18}) we may now conclude 
\begin{align} 
\begin{split} 
\int_{\partial\Omega}(J(y,\cdot))^{pq}\, dy<\infty\quad\mbox{\rm if} 
\hphantom{aaaaaaaaaaaaa}\nonumber \\
\int_{\partial\Omega}(J(y,\cdot))^q\, dy<\infty\quad\mbox{\rm and}\quad 
\int_{\Omega}\|g(z,\cdot )\|^q_{L^1(V)}\, dz<\infty.
\end{split} 
\end{align}
Together with (\ref{4.10}) the last line says that we have accomplished the 
execution step of the iteration. Summing up Steps 1 and 2, we have proved that 
\begin{align}\label{4.19}
\begin{split} 
\int_{\partial\Omega}(J(y,\cdot))^q\, dy<\infty\quad\mbox{\rm and}\quad\int_{ 
\Omega}\|g(z,\cdot )\|^q_{L^1(V)}\, dz<\infty
\end{split} 
\end{align}
for all $1\le q<\infty$.
\end{proof}

For the sake of completeness we provide a fact which may be known to experts, 
even in a more general form. It is a part of \cite{JK82}, Theorem 5.8. We mention 
that the measure $\omega\equiv\omega^{x_0}$ on $(\partial\Omega,{\mathcal B} 
(\partial\Omega))$ in this reference is the harmonic measure relative to $\Omega$ 
and some $x_0\in\Omega$. However since $\partial\Omega$ is smooth in our setting, 
here $\omega$ is equivalent to the Lebesgue surface measure on $\partial\Omega$. 
The Radon-Nikodym derivative is the Poisson kernel $\t k(x_0,\cdot)$ on $\partial 
\Omega$ which is bounded on $\partial\Omega$. For the Poisson kernel on $\Omega 
\times\partial\Omega$ we refer to \cite{Dy02}, Theorem 1.4 of Chapter 6, in 
particular (1.17). However we remark that in \cite{Dy02} the Poisson kernel is 
defined as the Radon-Nikodym derivative of the harmonic measures with respect to 
the normalized Lebesgue surface measure. 

For $\alpha>0$ introduce $\Gamma_\alpha(y):=\{x\in\Omega:|x-y|<(1+\alpha)\inf_{z\in 
\partial\Omega}|x-z|\}$. 
\begin{lemma}\label{Lemma4.2} (\cite{JK82}, Theorem 5.8)
Let $f\in L^1(\partial\Omega)$. Then there is a harmonic function $h_f$ on $\Omega$ 
such that $f(y)=\lim_{\Gamma_\alpha(y)\ni x\to y}h_f(x)$ for a.e. $y\in\partial\Omega$ 
and all $\alpha>0$. We have $h_f=\int_{\partial\Omega}f\, d\omega^x$, $x\in\Omega$. 
\end{lemma}

Let us continue with our analysis. 
\begin{lemma}\label{Lemma4.3} 
Let $g$ satisfy (\ref{3.2}) in the sense of Remark \ref{Remark3.2} and suppose 
(\ref{3.1}). The function $\partial\Omega\ni r\mapsto J(r,\cdot)(g)$ belongs to 
$L^\infty(\partial\Omega)$ and $\Omega\ni y\mapsto\|g(y,\cdot)\|_{L^1(V)}$ 
belongs to $L^\infty(\Omega)$. 
\end{lemma}
\begin{proof} For a.e. $r\in\partial\Omega$ we have by (\ref{4.11}), 
(\ref{4.12}), and (\ref{4.14}) 
\begin{eqnarray*}
&&\hspace{-.5cm}J(r,\cdot)(g)=\int_{\{v\in V:v\circ n(r)\ge 0\}}v\circ n(r)\, 
g(r,v)\, dv \\ 
&&\hspace{.5cm}\le M_{\rm max}\int_{\{v\in V:v\circ n(r)\ge 0\}}v\circ n(r)\, 
J(r^-(r,v),\cdot)(g)\, dv \\ 
&&\hspace{1.0cm}+\int_{\{v\in V:v\circ n(r)\ge 0\}}v\circ n(r)\int_0^{T_\Omega 
(r,v)}\lambda Q^+(g,g)(r-sv,v)\, ds\, dv \\ 
&&\hspace{.5cm}\le C_{M,w}\int_{y\in\partial\Omega}\, J(y,\cdot)(g)\, dy+C_{Q, 
w}\int_\Omega|r-y|^{1-d}\, \|g(y,\cdot)\|_{L^1(V)}\, dy  
\end{eqnarray*}
where 
\begin{eqnarray*}
C_{M,w}:=M_{\rm max}C_\Omega^2I_V(d+1)({\rm diam}(\Omega))^{3-d}\quad\mbox{\rm 
and}\quad C_{Q,w}:=c_Q\lambda I_V(d). 
\end{eqnarray*}
Together with (\ref{4.19}) it is immediate from H\"older's inequality and the 
boundedness of $\Omega$ that 
\begin{eqnarray}\label{4.20} 
\partial\Omega\ni r\mapsto J(r,\cdot)(g)\quad\mbox{\rm belongs to }L^\infty 
(\partial\Omega)\, .
\end{eqnarray}

Similar to (\ref{4.9}) it follows from (\ref{4.4}), the first line of (\ref{4.5}), 
and (\ref{4.8}) for $q=1$, i. e. $c_q=1$ in (\ref{4.4}) and (\ref{4.8}), that for 
a.e. $y\in \Omega$
\begin{align}\label{4.21}
\begin{split} 
&\hspace{-0.5cm}\|g(y,\cdot)\|_{L^1(V)}\le C_{M,v}\int_{r\in(\partial\Omega)_y} 
|y-r|^{1-d}\cdot\frac{n(r)\circ (r-y)}{|y-r|}J(r,\cdot)(g)\, dr \\ 
&\hspace{0.5cm}+C_{Q,v}\int_\Omega |y-x|^{1-d} \|g(x,\cdot)\|_{L^1(V)}\, dx\, . 
\end{split} 
\end{align}
Again by H\"older's inequality and the boundedness of $\Omega$, (\ref{4.19}) 
implies that the second item on the right-hand side of (\ref{4.21}) is bounded for 
$y\in\Omega$. 
\medskip

Regarding the first item on the right-hand side of (\ref{4.21}), we observe that 
for all $r\in\partial\Omega$ the function ${\mathbb R}^d\setminus\{r\}\ni y\mapsto 
h_r(y):=|y-r|^{1-d}\cdot n(r)\circ (r-y)/|y-r|$ is a harmonic function on 
${\mathbb R}^d\setminus\{r\}$. Furthermore, by (\ref{3.13}) we have 
\begin{align}\label{4.22}
\begin{split} 
&\hspace{-0.5cm}\int_{r\in\partial\Omega}|h_r(y)|\, dr \\
&\hspace{0.5cm}=\int_{r\in\partial\Omega}|y-r|^{1-d}\cdot\frac{|n(r)\circ (r-y)|} 
{|y-r|}\, dr \\ 
&\hspace{0.5cm}\le C_\Omega\sup_{r'\in\partial\Omega}\int_{r\in\partial\Omega} 
|r'-r|^{(2-d)}\, dr=:c_{\Omega,1}<\infty\, ,\quad y\in\partial\Omega.  
\end{split} 
\end{align}
With (\ref{4.20}) we obtain 
\begin{eqnarray}\label{4.23} 
\int_{r\in\partial\Omega}|h_r(y)|J(r,\cdot)(g)\, dr\le c_{\Omega,1}\cdot\esssup_{ 
r\in\partial\Omega}J(r,\cdot)(g)<\infty\, ,\quad y\in\partial\Omega.
\end{eqnarray}

According to (\ref{4.22}) we have $h_r|_{\partial\Omega}\in L^1(\partial\Omega)$, 
$r\in\partial\Omega$. It follows now from Lemma \ref{Lemma4.2} that there is a 
harmonic function $h_{r,\Omega}$ on $\Omega$ for which $\lim_{\Gamma_\alpha(y) 
\ni x\to y}h_{r,\Omega}(x)=|h_r(y)|$ for a.e. $y\in \partial\Omega$ for all 
$\alpha>0$. Since both functions $h_r$ as well as $h_{r,\Omega}$ are harmonic on 
$\Omega$ we can apply the representation via harmonic measures provided in Lemma 
\ref{Lemma4.2} to both. The boundary value functions of $h_r$ and $h_{r,\Omega}$ 
are, respectively, $h_r|_{\partial\Omega}$ and $|h_r(y)||_{\partial\Omega}$. Thus 
$h_r\le h_{r,\Omega}$ on $\Omega$ for all $r\in\partial\Omega$. The function  
\begin{eqnarray*} 
\Omega\ni y\mapsto\int_{r\in\partial\Omega}h_{r,\Omega}(y)J(r,\cdot)(g)\, dr=: 
H_\Omega(y) 
\end{eqnarray*}
is non-negative and harmonic on $\Omega$. Since $H_\Omega$ cannot have a local 
maximum on $\Omega$, it bounded by $c_{\Omega,1}\cdot\esssup_{r\in\partial\Omega} 
J(r,\cdot)(g)$ according to (\ref{4.23}). This allows us to conclude from 
$h_r\le h_{r,\Omega}$ on $\Omega$ that
\begin{eqnarray*}
&&\hspace{-0.5cm}\int_{r\in(\partial\Omega)_{y}}|y-r|^{1-d}\cdot\frac{n(r)\circ 
(r-y)}{|y-r|}J(r,\cdot)(g)\, dr \\
&&\hspace{0.5cm}=\int_{r\in(\partial\Omega)_{y}}h_r(y)J(r,\cdot)(g)\, dr \\
&&\hspace{0.5cm}\le\int_{r\in(\partial\Omega)_{y}}h_{r,\Omega}(y)J(r,\cdot)(g)\, 
dr \\
&&\hspace{0.5cm}\le H_\Omega(y)\le c_{\Omega,1}\cdot\esssup_{r\in\partial\Omega} 
J(r,\cdot)(g)\, ,\quad y\in\Omega.\vphantom{\int}
\end{eqnarray*}
In conclusion, both items on the right-hand side of (\ref{4.21}) have turned out to 
be bounded for a.e. $y\in\Omega$. In other words, 
\begin{eqnarray*} 
\Omega\ni y\mapsto \|g(y,\cdot)\|_{L^1(V)}\quad\mbox{\rm belongs to }L^\infty 
(\Omega).
\end{eqnarray*}
We have completed the proof. 
\end{proof}
\begin{theorem}\label{Theorem4.4} 
Let $g$ satisfy (\ref{3.2}) in the sense of Remark \ref{Remark3.2} and suppose 
(\ref{3.1}). We have 
\begin{eqnarray*}
g\in L^\infty(\Omega\times V)\quad{and}\quad\frac1g\in L^\infty(\Omega\times V) 
\, . 
\end{eqnarray*} 
\end{theorem}
\begin{proof} {\it Step 1 } In this step we verify $1/g\in L^\infty(\Omega\times 
V)$. For this let us mention that, by Lemma \ref{Lemma3.3}, we have $1\le\psi_g$ 
and $\sup\psi_g<\infty$ where the supremum is taken over $\{(r,v,t):(r,v)\in\Omega 
\times V,\ t\in [0,T_{\Omega}(r,v)]\}$. Recall also the definition of $\psi_g$ 
and (\ref{3.6}). 

We emphasize the following. Since for $(r,v)\in\partial\Omega\times V$ with 
$v\circ n(r)\ge 0$ we have $r^-\equiv r^-(r,v)\in\partial\Omega$ and $v\circ 
n(r^-)\le 0$, the boundary conditions (\ref{2.1}) say that 
\begin{eqnarray}\label{4.24} 
g(r^-,v)=J(r^-,\cdot)(g)\cdot M(r^-,v)\, . 
\end{eqnarray} 
For the next chain of equations and inequalities fix an $r\in\partial\Omega$ such 
that $J(r,\cdot)(g)<\infty$ and recall that $J(y,\cdot)(g)<\infty$ holds for a.e. 
$y\in\partial\Omega$ by Lemma \ref{Lemma4.3}. We obtain from (\ref{4.24})
\begin{align}\label{4.25} 
\begin{split} 
&\hspace{-.5cm}\int_{\{v\in V:v\circ n(r)\ge 0\}}v\circ n(r)\cdot 
J(r^-(r,v),\cdot)(g)\cdot M(r^-,v)\, dv \\ 
&\hspace{.5cm}=\int_{\{v\in V:v\circ n(r)\ge 0\}}v\circ n(r)\cdot 
g(r^-,v)\, dv \\ 
&\hspace{.5cm}=\int_{\{v\in V:v\circ n(r)\ge 0\}}v\circ n(r)\cdot 
g(r-T_{\Omega}(r,v)v,v)\, dv \\ 
&\hspace{.5cm}\le\sup\psi_g\cdot\int_{\{v\in V:v\circ n(r)\ge 0\}} 
v\circ n(r)\cdot g(r,v)\, dv \\ 
&\hspace{.5cm}=\sup\psi_g\cdot J(r,\cdot)(g)\vphantom{\int}  
\end{split} 
\end{align} 
where the second last line is a consequence of (\ref{3.6}) and the last 
equality sign holds according to the definition of $J$ in (i). Let us keep 
in mind that according to (ii), there exist constants $M_{\rm min},M_{\rm max} 
\in (0,\infty)$ such that $M_{\rm min}\le M(r,v)\le M_{\rm max}$. Instead of 
one fixed $r\in\partial\Omega$, let us now consider (\ref{4.25}) for a sequence 
$r_k\in\partial\Omega$, $k\in {\mathbb N}$, with $r_k\stack {k\to\infty}{\lra} 
r_\infty$ for some $r_\infty\in\partial\Omega$. 
\medskip

Assuming $J(r_k,\cdot)(g)\stack {k\to\infty}{\lra}0$ on the right-hand side of 
(\ref{4.25}), the left-hand side of (\ref{4.25}) implies that $J(y,\cdot)(g)=0$ 
for a.e. $y\in (\partial\Omega)_{r_\infty}$. This can be seen as follows. We take 
into consideration that for a.e. $y\in (\partial\Omega)_{r_\infty}$ there is a 
$k_0\equiv k_0(y)$ such that $y\in (\partial\Omega)_{r_k}$ for $k>k_0$ which means 
that $y\in\bigcap_{l\in {\mathbb N}}\bigcup_{k>l}(\partial\Omega)_{r_k}$. In 
addition, for the left-hand side of (\ref{4.25}) we get as in (\ref{4.12})
\begin{eqnarray*} 
&&\hspace{-.5cm}M_{\rm min}\cdot I_V(d+1)\int_{y\in(\partial\Omega)_{r_k}}J(y, 
\cdot )(g)\cdot\frac{(r_k-y)\circ n(r_k)\cdot n(y)\circ (y-r_k)}{|r_k-y|^{1+d}} 
\, dy \\ 
&&\hspace{.5cm}\le\int_{\{v\in V:v\circ n(r_k)\ge 0\}}v\circ n(r_k)\cdot J(r^- 
(r_k,v),\cdot)(g)\cdot M(r^-(r_k,v),v)\, dv \\ 
&&\hspace{-.0cm}\stack {k\to\infty}{\lra}0\vphantom{\int}
\end{eqnarray*}
where the weight functions 
\begin{eqnarray*} 
&&\hspace{-.5cm}y\in\partial\Omega\mapsto\sigma_{d}(r_k,y):=\frac{(r_k-y)\circ 
n(r_k)\cdot n(y)\circ (y-r_k)}{|r_k-y|^{1+d}}\cdot\chi_{(\partial\Omega)_{r_k}} 
(y)\, ,  
\end{eqnarray*}
$k\in {\mathbb N}\cup\{\infty\}$, are uniformly bounded on $\partial\Omega$ for 
$d=2,3$ by (\ref{3.13}) and therefore satisfy $\sigma_{d}(r_k,\cdot)\stack {k\to 
\infty}{\lra}\sigma_{d}(r_\infty,\cdot)$ in $L^1(\partial\Omega)$. We may now 
conclude that $\int_{y\in\partial\Omega}J(y,\cdot )(g)\cdot\sigma_{d}(r_\infty,
y)\, dy$ $=0$, i. e. $J(y,\cdot)(g)=0$ for a.e. $y\in (\partial\Omega)_{r_\infty}$. 
\medskip

Considering now the right-hand side of (\ref{4.25}) for $y\in (\partial\Omega 
)_{r_\infty}$ instead of $r$, it follows from the left-hand side of (\ref{4.25}) 
that even $J(y',\cdot)(g)=0$ a.e. on $y'\in (\partial\Omega)_y$. By iteration of 
the last conclusion we obtain $J(y,\cdot)(g)=0$ a.e. on $y\in\partial\Omega$. 

The latter would imply $g=0$ a.e. on $\partial\Omega\times V$, recall (\ref{2.1}) 
and (i). From (\ref{3.6}) we would a.e. on $\{(r,v)\in\partial\Omega\times V:v 
\circ n(r)\ge 0\}$ obtain 
\begin{eqnarray*}
0=\psi_g(r,v,T_{\Omega}(r,v))\cdot\int_0^{T_{\Omega}}\frac{\lambda Q^+(g,g) 
(r-sv,v)}{\psi_g(y,v,s)}\, ds\, . 
\end{eqnarray*} 
Since $\psi_g\ge 1$ this would say $Q^+(g,g)=0$ a.e. on $\Omega\times V$. 
Together with $g=0$ a.e. on $\partial\Omega\times V$, by (\ref{3.6}) this would 
yield $g=0$ a.e. on $\Omega\times V$. 
\medskip

Consequently, the above formulated assumption cannot hold, which means that 
\begin{eqnarray*}
\inf_{r\in\partial\Omega}J(r,\cdot)(g)>0\, . 
\end{eqnarray*} 
From here and (\ref{2.1}) as well as $0<M_{\rm min}\le M$ we may now conclude 
\begin{eqnarray}\label{4.26}
\inf\{g(y,v):(y,v)\in\partial\Omega\times V:v\circ n(y)\le 0\}>0.  
\end{eqnarray} 
On the other hand, relation (\ref{3.6}) implies for all $(r,v)\in\partial\Omega 
\times V$ with $v\circ n(r)\ge 0$ such that $g(r,v)<\infty$ and $y:=r^-(r,v)$ 
\begin{align}\label{4.27} 
\begin{split}
&\hspace{-.5cm}g(y,v)=g(r-T_{\Omega}v,v) \\ 
&\hspace{.5cm}=\psi_g(r,v,T_{\Omega})\left(-\int_0^{T_{\Omega}}\frac{\lambda 
Q^+(g,g)(r-sv,v)}{\psi_g(r,v,s)}\, ds+g(r,v)\right). 
\end{split} 
\end{align} 
Since $(y,v)\in\partial\Omega\times V$ in (\ref{4.27}) satisfies $v\circ n(y) 
\le 0$ and $\sup\psi_g<\infty$ we obtain from (\ref{4.26}) and (\ref{4.27}) 
\begin{eqnarray*}
\essinf_{(r,v)\in\partial\Omega\times V:v\circ n(r)\ge 0}\left\{-\int_0^{T_{ 
\Omega}(r,v)}\frac{\lambda Q^+(g,g)(r-sv,v)}{\psi_g(r,v,s)}\, ds+g(r,v)\right 
\}>0. 
\end{eqnarray*} 
Now (\ref{3.6}) implies positivity of $\essinf g$, i.e. $1/g\in L^\infty 
(\Omega)\times V$. 
\medskip

\noindent
{\it Step 2 } In this step we verify $g\in L^\infty(\Omega\times V)$. Recalling 
(\ref{3.6}) and $1\le\psi_g$ we may state that for a.e. $(y,v)\in\Omega\times V$ 
and $r=y^-(y,v)$ it holds that 
\begin{eqnarray*}
&&\hspace{-.5cm}g(y,v)\le g(r,v)+\int_0^{T_{\Omega}(y,v)}\lambda Q^+(g,g) 
(y-sv,v)\, ds\, . 
\end{eqnarray*} 
Applying now (\ref{2.1}) together with (ii) and (\ref{3.10}),  we verify  
\begin{align}\label{4.28}
\begin{split} 
&\hspace{-.5cm}g(y,v)\le M_{\rm max}J(r,\cdot)(g)+\frac{c_Q\lambda}{v_{min}} 
\int_0^{T_{\Omega}(y,v)}\|g(y-sv,\cdot)\|_{L^1(V)}|v|\, ds \\ 
&\hspace{.5cm}=M_{\rm max}J(r,\cdot)(g)+\frac{c_Q\lambda}{v_{min}}\int_0^{ 
T_{\Omega}(y,v/|v|)}\|g(y-sv/|v|,\cdot)\|_{L^1(V)}\, ds \\ 
&\hspace{.5cm}\le M_{\rm max}\cdot\esssup_{r\in\partial\Omega}J(r,\cdot)(g) 
\vphantom{\int} \\ 
&\hspace{1.0cm}+\frac{c_Q\lambda}{v_{min}}\, {\rm diam}(\Omega)\cdot\esssup_{ 
y\in\Omega}\|g(y,\cdot)\|_{L^1(V)}<\infty\quad\mbox{\rm for a.e. }(y,v) 
\in\Omega\times V, 
\end{split} 
\end{align}
the last inequality because of Lemma \ref{Lemma4.3}. We have completed the proof 
of the theorem. 
\end{proof} 

\bibliographystyle{siamplain}

\begin{thebibliography}{} 
\bibitem{BG17}{\sc A. Bobylev and I. M. Gamba}, {\it Upper Maxwellian bounds for 
the Boltzmann equation with pseudo-Maxwell molecules}, Kinet. Relat. 
Models, 10 (2017), pp. 573-585. 

\bibitem{Br11}{\sc H. Brezis}, {\it Functional analysis, Sobolev spaces and partial 
differential equations}, Springer, New York, 2011. 

\bibitem{CPW98}{\sc S. Caprino, M. Pulvirenti, and W. Wagner}, {\it Stationary 
particle systems approximating stationary solutions to the Boltzmann 
equation}, SIAM J. Math. Analysis, 29 (1998), pp. 913-934. 

\bibitem{DMRT13}{\sc C. De Lellis, R. A. Minlos, U. Rehmann, and B. Tsirelson}, 
{\it Convergence of measures}, Encyclopedia of Mathematics, 
{\tt https://www.encyclopediaofmath.org/} {\tt index.php/Convergence{\_}of{\_}measures},  
Springer and EMS, 2013. 

\bibitem{Dy02}{\sc E. B. Dynkin}, {\it Diffusions, superdiffusions and partial 
differential equations}, American Mathematical Society Colloquium Publications, 
50, American Mathematical Society, Providence, RI, 2002. 

\bibitem{Fo01}{\sc N. Fournier}, {\it Strict positivity of the solution to a 
2-dimensional spatially homogeneous Boltzmann equation without cutoff}, 
Ann. Inst. H. Poincaré Probab. Statist., 37 (2001), pp. 481-502. 

\bibitem{GBV09}{\sc I. M. Gamba, V. Panferov, and C. Villani}, {\it Upper 
Maxwellian bounds for the spatially homogeneous Boltzmann equation}, 
Arch. Ration. Mech. Anal., 194 (2009), pp. 253-282. 

\bibitem{Ho03}{\sc L. H\"ormander}, {\it The analysis of linear partial 
differential operators I, Distribution theory and Fourier analysis}, 
Springer, Berlin, 2003. 

\bibitem{Ja97}{\sc S. Janson}, {\it Gaussian Hilbert spaces}, Cambridge 
University Press, Cambridge, 1997. 

\bibitem{JK82}{\sc D. S. Jerison and C. E. Kenig}, {\it Boundary behavior 
of harmonic functions in nontangentially accessible domains}, Adv. in Math., 
46 (1982), pp. 80–147.

\bibitem{Li94}{\sc J.-L. Lions}, {\it Compactness in Boltzmann's equation via 
Fourier integral operators and applications I}, J. Math. Kyoto Univ., 34 (1994), 
pp. 391-427. 

\bibitem{Lu98}{\sc X. Lu}, {\it A direct method for the regularity of the gain 
term in the Boltzmann equation}, J. Math. Anal. Appl., 228 (1998), pp. 409-435. 

\bibitem{Mo05}{\sc C. Mouhot}, {\it Quantitative lower bounds for the full Boltzmann 
equation. I. Periodic boundary conditions}, Comm. Partial Differential Equations, 
30 (2005), pp. 881-917. 

\bibitem{MV04}{\sc C. Mouhot and C. Villani}, {\it Regularity theory for the 
spatially homogeneous Boltzmann equation with cut-off}, Arch. Ration. Mech. Anal., 
173 (2004), pp. 169-212. 
\end{thebibliography}

\end{document}